%% file: acyl-final-revised-4-1-2015.tex
\documentclass[11pt,letterpaper]{article}
\usepackage{amsfonts, amsmath, amssymb, amscd, amsthm, graphicx, mathrsfs, wasysym, setspace, mdwlist, color}
\usepackage{hyperref}
\makeatletter
\def\blfootnote{\xdef\@thefnmark{}\@footnotetext}
\makeatother

\newcommand{\mc}[1]{\mathcal{{#1}}}

\newtheorem{thm}{Theorem}[section]
\newtheorem{cor}[thm]{Corollary}
\newtheorem{lem}[thm]{Lemma}
\newtheorem{prop}[thm]{Proposition}
\newtheorem{prob}[thm]{Question}

\theoremstyle{definition}
\newtheorem{defn}[thm]{Definition}
\theoremstyle{remark}
\newtheorem{rem}[thm]{Remark}
\newtheorem{ex}[thm]{Example}

\newcommand{\G }{\Gamma (G, X\sqcup \mathcal H)}
\newcommand{\Gy}{\Gamma (G, Y\sqcup \mathcal H)}
\newcommand{\Gz}{\Gamma (G, Z\sqcup \mathcal H)}
\newcommand{\dxh }{{\rm d}_{X\cup\mathcal H}}
\newcommand{\dyh }{{\rm d}_{Y\cup\mathcal H}}
\newcommand{\dl }{\widehat{\rm d}_{\lambda}}

\newcommand{\e }{\varepsilon }
\renewcommand{\kappa }{\varkappa}

\newcommand{\Hl }{\{ H_\lambda \} _{\lambda \in \Lambda } }

\renewcommand{\d }{{\rm d} }

\newcommand{\he }{hyperbolically embedded }
\newcommand{\Lab }{{\bf Lab}}
\newcommand{\h}{\hookrightarrow _{h}}

\newcommand{\QZ}{{QZ^1}}

\begin{document}

\title{Acylindrically hyperbolic groups}
\author{D. Osin\thanks{This work was supported by the NSF grant DMS-1006345 and by the RFBR grant 11-01-00945.}}
\date{}

\maketitle

\begin{abstract}
We say that a group $G$ is \emph{acylindrically hyperbolic } if it admits a non-elementary acylindrical action on a hyperbolic space. We prove that the class of acylindrically hyperbolic groups coincides with many other classes studied in the literature, e.g., the class $C_{geom}$ introduced by Hamenst\"adt, the class of groups admitting a non-elementary weakly properly discontinuous action on a hyperbolic space in the sense of Bestvina and Fujiwara, and the class of groups with hyperbolically embedded subgroups studied by Dahmani, Guirardel, and the author. We also record some basic results about acylindrically hyperbolic groups for future use.
\end{abstract}

\vspace{-2mm} \hspace{3mm} \textbf{MSC Subject Classification:} 20F67, 20F65.

\tableofcontents


\section{Introduction}


The action of a group $G$ on a metric space $S$ is called {\it acylindrical} if for every $\e>0$ there exist $R,N>0$ such that for every two points $x,y$ with $\d (x,y)\ge R$, there are at most $N$ elements $g\in G$ satisfying
$$
\d(x,gx)\le \e \;\;\; {\rm and}\;\;\; \d(y,gy) \le \e.
$$
(By default, all actions are assumed to be isometric in this paper.) Informally, one can think of this condition as a kind of properness of the action on $S\times S$ minus a ``thick diagonal". The notion of acylindricity goes back to Sela's paper \cite{Sel}, where it was considered for groups acting on trees. In the context of general metric spaces, the above definition is due to Bowditch \cite{Bow}.

In the recent years, many interesting results were obtained for groups that admit a non-elementary action on a hyperbolic space which is acylindrical or satisfies certain similar assumptions such as weak acylindricity introduced by Hamenst\"adt \cite{Ham}, weak proper discontinuity introduced by Bestvina and Fujiwara \cite{BF}, or existence of weakly contracting elements in the sense of Sisto \cite{Sis}. Groups acting acylindrically and non-elementary on hyperbolic spaces also serve as the main source of examples in the paper \cite{DGO}, where some parts of the theory of relatively hyperbolic groups were generalized in a more general context of groups with hyperbolically embedded subgroups.

Our main goal is to show that the classes of groups considered in the above-mentioned papers are essentially the same and coincide with the class of acylindrically hyperbolic groups defined below. This paper is also aimed to serve as a reference for basic properties of acylindrically hyperbolic groups. Finally we include a brief survey, which brings together some known examples and results about acylindrically hyperbolic groups proved in \cite{BBF,BF,DGO,Ham,HO,Sis} and other papers under various assumptions equivalent to acylindrical hyperbolicity.

Let us recall the standard terminology. Given a group $G$ acting on a hyperbolic space $S$, an element $g\in G$ is called \emph{elliptic} if some (equivalently, any) orbit of $g$ is bounded, and \emph{loxodromic} if the map $\mathbb Z\to S$ defined by $n\mapsto g^ns$ is a quasi-isometry for some (equivalently, any) $s\in S$. Every loxodromic element $g\in G$ has exactly $2$ limit points $g^{\pm\infty}$ on the Gromov boundary $\partial S$. Loxodromic elements $g,h\in G$ are called \emph{independent} if the sets $\{ g^{\pm \infty}\}$ and $\{ h^{\pm \infty}\}$ are disjoint.

We begin with a classification of groups acting acylindrically on hyperbolic spaces.

\begin{thm}\label{class}
Let $G$ be a group acting acylindrically on a hyperbolic space. Then $G$ satisfies exactly one of the following three conditions.
\begin{enumerate}
\item[(a)] $G$ has bounded orbits.
\item[(b)] $G$ is virtually cyclic and contains a loxodromic element.
\item[(c)] $G$ contains infinitely many independent loxodromic elements.
\end{enumerate}
\end{thm}

Applying the theorem to cyclic groups, we recover the following result of Bowditch \cite{Bow}:  every element of a group acting acylindrically on a hyperbolic space is either elliptic or loxodromic. Compared to the general classification of groups acting on hyperbolic spaces, Theorem \ref{class} rules out parabolic and quasi-parabolic actions in Gromov's terminology \cite{Gro}. The non-trivial part is to show that $G$ cannot act parabolically. Then the theorem follows from the known fact that if $G$ is not virtually cyclic and contains a loxodromic element, then it is of type (c). This was proved by Bestwina--Fujiwara \cite{BF} for weakly properly discontinuous actions and Hamenst\"adt \cite{Ham} for weakly acylindrical actions. Note that acylindricity is essential here; it is easy to see that Theorem \ref{class} can fail in various ways even for (non-acylindrical) proper actions. For more details we refer to Section \ref{Secclass}.

We now state the main result of this paper. Recall that an action of a group $G$ on a hyperbolic space $S$ is called \emph{elementary} if the limit set of $G$ on $\partial S$ contains at most $2$ points. If the action is acylindrical,  non-elementarity is equivalent to condition (c) from Theorem \ref{class}.

\begin{thm}\label{main}
For any group $G$, the following conditions are equivalent.
\begin{enumerate}
\item[(AH$_1$)] There exists a generating set $X$ of $G$ such that the corresponding Cayley graph $\Gamma (G,X)$ is hyperbolic,  $|\partial \Gamma (G,X)|> 2$, and the natural action of $G$ on $\Gamma (G,X)$ is acylindrical.
\item[(AH$_2$)] $G$ admits a non-elementary acylindrical action on a hyperbolic space.
\item[(AH$_3$)] $G$ is not virtually cyclic and admits an action on a hyperbolic space such that at least one element of $G$ is loxodromic and satisfies the WPD condition.
\item[(AH$_4$)] $G$ contains a proper infinite hyperbolically embedded subgroup.
\end{enumerate}
\end{thm}

Here WPD is the abbreviation of the Bestvina--Fujiwara weak proper discontinuity condition; for a loxodromic element $g$ it requires the action of $G$ to be acylindrical in the direction of a (quasi) axis of $g$. For the precise definition of this condition as well as for the definition of hyperbolically embedded subgroups we refer to the next section. Note that we do not assume that the actions in (AH$_2$) and (AH$_3$) are cobounded, while the action in (AH$_1$) obviously is. We are not aware of any elementary construction which allows to get cobounded actions on hyperbolic spaces from non-cobounded ones and preserves non-elementarity and acylindricity.

It immediately follows from definitions that (AH$_1$) $\Longrightarrow$ (AH$_2$) $\Longrightarrow$ (AH$_3$). The implication (AH$_3$) $\Longrightarrow$ (AH$_4$) is non-trivial and was proved in \cite{DGO}. Thus we only need to prove that (AH$_4$) $\Longrightarrow$ (AH$_1$). In fact, we prove a stronger statement, Theorem \ref{XtoY}, which seems to be of independent interest. The proof makes use of some technical tools from the paper \cite{HO}.

\begin{defn}
We call a group $G$ \emph{acylindrically hyperbolic} if it satisfies either of the equivalent conditions (AH$_1$)--(AH$_4$) from Theorem \ref{main}.
\end{defn}

Note that every group has an obvious acylindrical action on a hyperbolic space, namely the trivial action on a point. Thus considering elementary acylindrically hyperbolic groups does not make much sense. For this reason we include non-elementarity in the definition.

Many other conditions studied in the literature are obviously intermediate between some of the conditions from Theorem \ref{main} and hence are equivalent to acylindrical hyperbolicity. One example is the weak acylindricity in the sense of Hamenst\"adt \cite{Ham}, which is weaker than acylindricity, but stronger than weak proper discontinuity. This implies that the class $C_{geom}$ introduced in \cite{Ham} coincides with the class of acylindrically hyperbolic groups.

The class of acylindrically hyperbolic groups includes many examples of interest: non-elementary hyperbolic and relatively hyperbolic groups, all but finitely many mapping class groups of punctured closed surfaces, $Out(F_n)$ for $n\ge 2$, directly indecomposable right angled Artin groups, $1$-relator groups with at least $3$ generators, most $3$-manifold groups, and many other examples. On the other hand, acylindrical hyperbolicity is strong enough to imply non-trivial theorems. For a brief survey  of known examples and results about acylindrically hyperbolic groups we refer to Section 8.

The next theorem allows us to define a natural notion of a \emph{generalized loxodromic element} of a group in Section 4. Similarly to Theorem \ref{main}, our contribution amounts to proving that (L$_4$) implies (L$_1$). Indeed the implications (L$_1$) $\Longrightarrow$ (L$_2$),  (L$_2$) $\Longrightarrow$ (L$_3$) are immediate from the definitions, and (L$_3$) $\Longrightarrow$ (L$_4$) is proved in \cite{DGO}.

\begin{thm}\label{lox}
For any group $G$ and any $g\in G$, the following conditions are equivalent.
\begin{enumerate}
\item[(L$_1$)] There exists a generating set $X$ of $G$ such that the corresponding Cayley graph $\Gamma (G,X)$ is hyperbolic, the natural action of $G$ on $\Gamma (G,X)$ is acylindrical, and $g$ is loxodromic.
\item[(L$_2$)] There exists an acylindrical action of $G$ on a hyperbolic space such that $g$ is loxodromic.
\item[(L$_3$)] There exists an action of $G$ on a hyperbolic space such that $g$ acts loxodromically and satisfies the WPD condition.
\item[(L$_4$)] The order of $g$ is infinite and $g$ is contained in a virtually cyclic hyperbolically embedded subgroup of $G$.
\end{enumerate}
\end{thm}

Theorems \ref{class}, \ref{main}, and \ref{lox} have several almost immediate corollaries. We mention three examples here. Other applications can be found in \cite{Hull,MO}

Recall that a subgroup $H\le G$ is \emph{$s$-normal} in $G$ if $|H^g\cap H|=\infty $ for every $g\in G$. The corollary below is useful for showing that certain groups are not acylindrically hyperbolic (see Section 7). It generalizes \cite[Lemma 8.11]{DGO}; in our language, the latter results claims that the class of acylindrically hyperbolic groups is closed under taking normal subgroups. It is worth noting that our proof essentially uses acylindrical actions and is much shorter than the proof of Lemma 8.11 in \cite{DGO}.

\begin{cor}\label{s-norm}
The class of acylindrically hyperbolic groups is closed under taking $s$-normal subgroups.
\end{cor}

For a subgroup $H$ of a group $G$, the \emph{commensurator} of $H$ is defined by $$Comm_G(H)=\{ g\in G \mid [H: H\cap H^g]<\infty \; {\rm and} \; [H^g: H\cap H^g]<\infty\}.$$ The study of commensurators is partially motivated by the result of Mackey \cite{Mac} stating that $H=Comm_G(H)$ if and only if the quasi-regular representation of $G$ on $\ell^2(G/H)$ is irreducible. For more details we refer the reader to \cite{BdH}.

As usual, we say that a group acting on a metric space is \emph{elliptic} if it has bounded orbits. From Theorem  \ref{class} and Corollary \ref{s-norm}, we obtain the following.

\begin{cor}\label{ell}
For every group $G$ acting acylindrically on a hyperbolic space, the following hold.
\begin{enumerate}
\item[(a)] Every elliptic subgroup of $G$ is contained in a maximal elliptic subgroup of $G$.
\item[(b)] For every infinite maximal elliptic subgroup $H\le G$, we have $H=Comm_G(H)$.
\end{enumerate}
\end{cor}

Finally we show the following.

\begin{prop}\label{bgen}
Let $G$ be an acylindrically hyperbolic group. Suppose that $G=G_1\ldots G_n$ for some subgroups $G_1, \ldots, G_n$ of $G$. Then $G_i$ is acylindrically hyperbolic for at least one $i$.
\end{prop}

If $G=G_1\ldots G_n$, then one says that $G$ is \emph{boundedly generated} by subgroups $G_1, \ldots, G_n$. It is known that an acylindrically hyperbolic group $G$ cannot be boundedly generated by cyclic subgroups. Indeed this easily follows from the result of Bestvina-Fujiwara \cite{BF} about infiniteness of the dimension of the space of non-trivial quasimorphisms on $G$; for details we refer to the discussion after Theorem 2.30 in \cite{DGO}. More generally, the same result can be derived for amenable subgroups $G_1, \ldots, G_n$ using quasi-cocycles with coefficients in $\ell ^2(G)$. (This is just a combination of results from \cite{Thom} and \cite{Ham} or \cite{HO}). Note that every acylindrically hyperbolic group is non-amenable since it contains non-abelian free subgroups by the standard ping-pong argument. Thus our Proposition \ref{bgen} can be thought of as a generalization of the above-mentioned results.

The paper is organized as follows. The next section contains some preliminary information on hyperbolic spaces, group actions, and hyperbolically embedded subgroups. In Section 3, we discuss the general classification of group actions on hyperbolic spaces and prove Theorem \ref{class}. In Section 4, we review and generalize some technical tools from \cite{HO} necessary for the proof of Theorem \ref{main}; the proof itself is given in Section 5. Section 6 is devoted to loxodromic elements and the proof of Theorem \ref{lox}. Corollaries \ref{s-norm} and \ref{ell}, Proposition \ref{bgen}, and other applications are discussed in Section 7. Finally in Section 8, we give a brief survey of known examples and results about acylindrically hyperbolic groups.

{\bf Acknowledgment.} I am grateful to Francois Dahmani, Ashot Minasyan, Alexander Olshanskii, and many other colleagues with whom I discussed various topics related to this paper. I am also grateful to the anonymous referee for useful remarks. Finally, I would like to thank my students Sahana Balasubramanya and Bryan Jacobson for careful reading of the manuscript and pointing out numerous misprints and inaccuracies.


\section{Preliminaries}


\paragraph{2.1. General notation and conventions.} All generating sets considered in this paper are supposed to be symmetric, i.e., closed under taking inverse elements. If $G$ is a group generated by  $X$, we write $U\equiv W$ for two words $U, W$ in $X$ if they are equal as words and $U=_GV$ if they represent the same element of $G$. The word length $|\cdot |_X$ on a group $G$ corresponding to a (not necessary generating) set $X$ is defined by letting $|g|_X$ be the length of a shortest word in $X\cup X^{-1}$ representing $g$ if $g\in \langle X\rangle $ and $|g|_X=\infty $ otherwise. The corresponding metric on $G$ is denoted by $\d_S$; thus $\d_S(f,g)=|f^{-1}g|_S$.
Further we denote by $\Gamma (G, X)$ the corresponding Cayley graph. By a path $p$ in a Cayley graph we always mean a combinatorial path; we denote the label of $p$ by $\Lab (p)$ and we denote the origin and terminus of $p$ by $p_-$ and $p_+$, respectively. The length of $p$ is denoted by $\ell (p)$.

Given a group $G$ acting on a space $S$ and a subset $A\le S$, we denote by $gA$ the image of  $A$ under the action of an element $g\in G$. For two points $x,y\in S$, $[x,y]$ denotes a geodesic going from $x$ to $y$ in $S$.

\paragraph{2.2. Hyperbolic spaces.} We employ the definition of a  $\delta$-hyperbolic space via the Rips condition. That is, a metric space $S$ is $\delta$-hyperbolic if it is geodesic and for any geodesic triangle $\Delta $ in $S$, every side of $\Delta $ is contained in the union of the $\delta$-neighborhoods of the other two sides.

Recall that the \emph{Gromov product} of points $x,y$ with respect to a point $z$ in a metric space $(S,d)$ is defined by
$$
(x,y)_z=\frac12( \d(x,z) +\d (y,z) -\d (x,y)).
$$

We list here some results about hyperbolic spaces used in this paper. The first one is is a particular case of Theorem 16 in Chapter 5 of \cite{GH}. It is similar to the lemma proved in Sec. 7.2C of \cite{Gro} (see also \cite[Lemma 5]{IO}).

\begin{lem}\label{GH}
Let $(S,d)$ be a $\delta$-hyperbolic metric space, $s_0, \ldots , s_n$ a sequence of points in $S$ such that
\begin{equation}\label{dsi}
\d (s_{i-1}, s_{i+1})\ge \max \{ \d (s_{i-1}, s_i), \d (s_{i+1}, s_i)\} +18\delta +1
\end{equation}
for every $i\in \{ 1, \ldots, n-1\}$. Then $\d (s_0, s_n)\ge n$.
\end{lem}

\begin{rem}\label{dsieq}
It is sometimes convenient to rewrite (\ref{dsi}) in the following equivalent form using Gromov products:
$$
2(s_{i-1}, s_{i+1})_{s_i} \le \min \{ \d (s_{i-1}, s_i), \d (s_{i+1}, s_i)\} -18\delta -1.
$$
\end{rem}

The next lemma is well-known (see the proof of Proposition 21 and Chapter 2 of \cite{GH}). It relates various definitions of a hyperbolic space.
\begin{lem}\label{GHs23}
Let $S$ be a $\delta$-hyperbolic space. Then the following hold.
\begin{enumerate}
\item[(a)] For any $x,y,z,t\in S$, we have
\begin{equation}\label{Gprod}
(x,z)_t\ge \min \{ (x,y)_t, (y,z)_t\} -8\delta.
\end{equation}
\item[(b)] For any $x,y,z\in S$ and any $u\in [x,y]$, $v\in [x,z]$ such that $\d(x,u)=\d(x,v)\le (y,z)_x$, we have $\d(u,v)\le 4\delta$.
\end{enumerate}
\end{lem}

By $\partial S$ we denote the Gromov boundary of $S$. Since we do not assume that $S$ is proper, the boundary is defined as the set of equivalence classes of sequences convergent at infinity. More precisely, a sequence $(x_n)$ of elements of $S$ converges at infinity if $(x_i, x_j)_s\to \infty $ as $i,j\to \infty$ (this definition is clearly independent of the choice of $s$). Two such sequences $(x_i)$ and $(y_i)$ are equivalent if $(x_i,y_j)_s\to \infty$ as $i,j\to \infty$. If $a$ is the equivalence class of $(x_i)$, we say that the sequence $x_i$ converges to $a$. This defines a natural topology on $S\cup \partial S$ with respect to which $S$ is dense in $S\cup \partial S$.

\paragraph{2.3. Acylindricity and the Bestvina-Fujiwara WPD condition.}
The following result will be used several times in this paper.

\begin{lem}\label{=R}
The action of a group $G$ on a hyperbolic space $S$ is acylindrical if and only if for every $\e>0$ there exist $R,N>0$ such that for every two points $x,z$ satisfying $\d (x,z)=R$, we have
$$
\sharp \{ g\in G\mid \max\{ \d(x,gx), \d(y,gz)\} \le \e\} \le N.
$$
\end{lem}

\begin{figure}
 \centering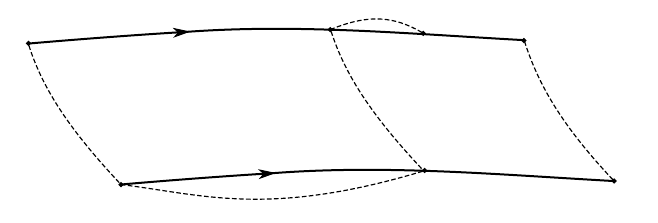\\
  \caption{}\label{fig01}
\end{figure}

\begin{proof}
The ``only if" part of the claim is obvious. Let us prove the ``if" part. Suppose that $S$ is $\delta$-hyperbolic. Fix any $\e>0$. By our assumption, there exist $R, N>0$ such that for every two points $x,z$ satisfying $\d (x,z)=R$, we have
\begin{equation}\label{xz}
| \{ g\in G\mid \max\{ \d(x,gx), \d(y,gz)\} \le 4\delta +3\e\} | \le N.
\end{equation}

Let now $\d(x, y)\ge R$ and let $p$ be a geodesic in $S$ connecting $x$ to $y$. Let $z$ be the point on $p$ such that $\d (x,z)=R$. Let $A$ be the subset consisting of all $g\in G$ satisfying $\max\{ \d(x,gx), \d(y,gy)\} \le \e$. Consider any $g\in A$ and let $q=gp$; thus $q$ is a geodesic connecting $gx$ to $gy$ (see Fig. \ref{fig01}). Drawing a diagonal in the geodesic quadrangle $[x,gx]q[gy,y]p^{-1}$ and applying the Rips condition twice, we obtain that $\d (z, [x,gx]q[gy,y])\le 2\delta $. Since $\ell ([x,gx])\le \e$ and $\ell ([gy,y])\le \e$, we obtain $\d (z, q) \le 2\delta + \e$. Let $z_0$ be a point on $q$ such that $$\d (z,z_0)\le 2\delta +\e.$$ Since $gx$, $z_0$, and $gz$ belong to the same geodesic $q$, we have
$$
\d (z_0, gz)  =  |\d (gx, z_0) - \d (gx, gz)| =  | \d (gx, z_0)- \d (x, z)|  \le  \d (x,gx) + \d (z, z_0)\le 2\delta + 2\e.
$$
Hence $$\d (z, gz)\le \d (z, z_0)+\d (z_0, gz) \le 4\delta +3\e.$$ By (\ref{xz}), we obtain $|A|\le N$.
\end{proof}

The following definition is due to Bestvina and Fujiwara \cite{BF}.

\begin{defn}\label{WPD}
Let $G$ be a group acting on a hyperbolic space $S$, $h$ an element of $G$.  One says that $h$ satisfies the {\it weak proper discontinuity} condition (or $h$ is a {\it WPD element}) if for every $\e >0$ and every $x\in S$, there exists $M\in \mathbb N$ such that
\begin{equation}\label{eq: wpd}
| \{ g\in G \mid \d (x, g(x))<\e, \;   \d (h^M(x), gh^M(x))<\e \} |<\infty .
\end{equation}
If every loxodromic element satisfies the WPD condition, one says that $G$ acts on $S$ \emph{weakly properly discontinuously}. Obviously this is the case if $G$ acts on $S$ acylindrically.
\end{defn}

In this paper we will use the following result, which follows immediately from \cite[Proposition 6]{BF}. It was also proved by Hamenst\"adt for acylindrical actions \cite{Ham}.

\begin{lem}[Bestvina--Fujiwara]\label{wpd-el}
Let $G$ be a group acting weakly properly discontinuously on a hyperbolic space and containing a loxodromic element. Then either $G$ is virtually cyclic, or contains infinitely many independent loxodromic elements.
\end{lem}

\paragraph{2.4.  Hyperbolically embedded subgroups}
Let $G$ be a group with a fixed collection of subgroups $\Hl $.
Given a subset $X\subseteq G$ such that $G$ is generated by $X$ together with the union of all $H_\lambda$'s,  we denote by $\G $ the Cayley graph of $G$ whose edges are labeled by letters from the alphabet $X\sqcup\mathcal H$, where
\begin{equation}\label{calH}
\mathcal H= \bigsqcup\limits_{\lambda \in \Lambda } H_\lambda.
\end{equation}
That is, two vertices $g,h\in G$ are connected by an edge going from $g$ to $h$ and labeled by $a\in X\sqcup\mathcal H$ iff $a$ represents the element $g^{-1}h$ in $G$.

\begin{rem}\label{disj}
It is important that the unions in the definition above are disjoint. For example, it means that for every $h\in H_\lambda \cap H_\mu$, the alphabet $\mathcal H$ will have two letters representing the element $h$ in $G$: one in $\mathcal H_\lambda$ and the other in $\mathcal H_\mu$. It can also happen that a letter from $\mathcal H$ and a letter from $X$ represent the same element of $G$. If several letters from $X\sqcup\mathcal H$  represent the same element in $G$, then $\G $ has multiple edges corresponding to these letters. In particular, at every vertex of $\G$ we have a bunch of loops labelled by the letter $1\in H_\lambda$ for every $\lambda \in \Lambda$. We could avoid these redundant loops by using $H_\lambda \setminus\{ 1\}$ instead of $H_\lambda$ in (\ref{calH}), but they do not create any problems.
\end{rem}

In what follows, we think of the Cayley graphs $\Gamma (H_\lambda, H_\lambda  )$ as a complete subgraphs of $\G $.

\begin{defn}\label{dl}
For every $\lambda \in \Lambda $, we introduce a \textit{relative metric} $\dl \colon H_\lambda \times H_\lambda \to [0, +\infty]$ as follows. We say that a (combinatorial) path $p$ in $\G$ is {\it $\lambda$-admissible} (or simply \emph{admissible} if no confusion is possible) if it contains no edges of $\Gamma (H_\lambda, H_\lambda )$. Note that we do allow $p$ to pass through vertices of $\Gamma (H_\lambda, H_\lambda )$ as well as to have edges labelled by letters from $H_\lambda$ (provided these edges are not in $\Gamma (H_\lambda, H_\lambda )$). Let $\dl (h,k)$ denote the length of a shortest admissible path in $\G $ that connects $h$ to $k$. If no such a path exists, we set $\dl (h,k)=\infty $. Clearly $\dl $ satisfies the triangle inequality.
\end{defn}

\begin{figure}
  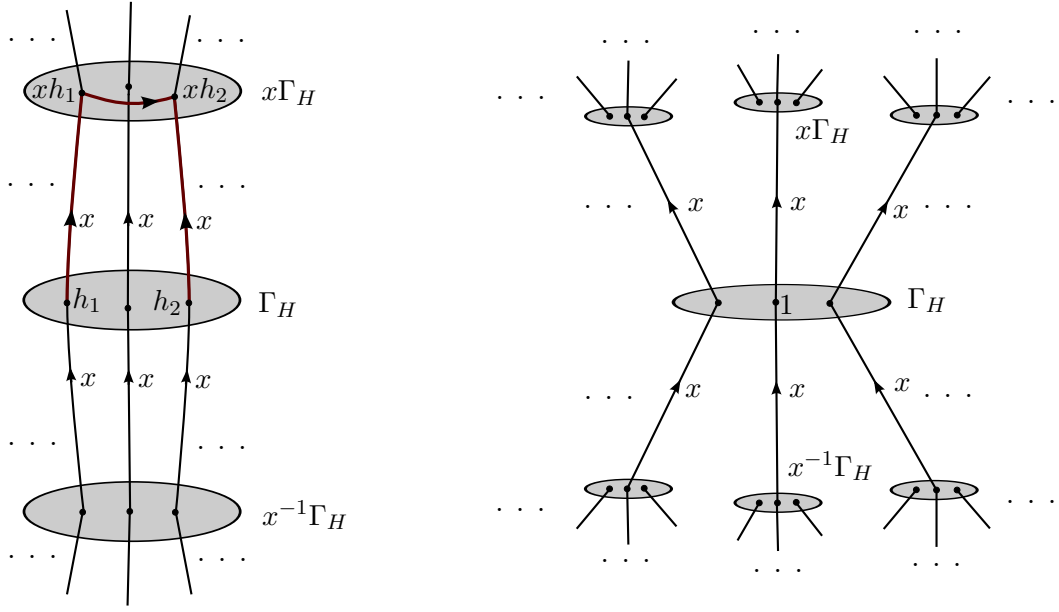
  \caption{Cayley graphs $\Gamma(G, X\sqcup H)$ for $G=H\times \mathbb Z$ and $G=H\ast \mathbb Z$.}\label{fig0}
\end{figure}

\begin{defn}\label{he-def}
Let $G$ be a group, $X$ a (not necessary finite) subset of $G$. We say that a collection of subgroups $\Hl$ of $G$ is \emph{\he in $G$ with respect to $X$} (we write $\Hl \h (G,X)$) if the following conditions hold.
\begin{enumerate}
\item[(a)] The group $G$ is generated by $X$ together with the union of all $H_\lambda$ and the Cayley graph $\G $ is hyperbolic.
\item[(b)] For every $\lambda\in \Lambda $, the metric space $(H_\lambda, \dl )$ is proper. That is, any ball of finite radius in $H_\lambda $ contains finitely many elements.
\end{enumerate}
Further we say  that $\Hl$ is \emph{hyperbolically embedded} in $G$ and write $\Hl\h G$ if $\Hl\h (G,X)$ for some $X\subseteq G$.
\end{defn}

Note that for any group $G$ we have $G\h (G, \emptyset)$.  Indeed in this case $X\sqcup \mathcal H=G$ and the Cayley graph $\Gamma(G, X\sqcup H)$ has diameter $1$. The corresponding relative metric satisfies $\widehat d(h_1, h_2)=\infty $ whenever $h_1\ne h_2$. Further, if $H$ is a finite subgroup of a group $G$, then $H\h (G, G)$. These cases are referred to as {\it degenerate}.

Since the notion of a hyperbolically embedded subgroup plays a crucial role in this paper, we consider two additional examples borrowed from \cite{DGO}.

\begin{ex}
\begin{enumerate}
\item[(a)]
Let $G=H\times \mathbb Z$ and let $X=\{ x\} $, where $x$ is a generator of $\mathbb Z$. Then $\Gamma (G, X\sqcup H)$ is quasi-isometric to $\mathbb R$ and hence it is hyperbolic. However the corresponding relative metric satisfies the inequality $\widehat\d(h_1, h_2)\le 3$ for every $h_1, h_2\in H$. Indeed let $\Gamma _H$ denote the Cayley graph $\Gamma (H, H) $. In the shifted copy $x\Gamma _H$ of $\Gamma _H$ there is an edge labeled by $h_1^{-1}h_2\in H$ that connects $h_1x$ to $h_2x$. Thus  there is an admissible path of length $3$ connecting $h_1$ to $h_2$ (see Fig. \ref{fig0}). If $H$ is infinite, this implies $H\not\h (G,X)$.

\item[(b)]  Let $G=H\ast \mathbb Z$, $X=\{ x\} $, where $x$ is a generator of $\mathbb Z$. In this case the Cayley graph $\Gamma (G, X\sqcup H)$ is quasi-isometric to a tree (see Fig. \ref{fig0}) and $\widehat\d(h_1, h_2)=\infty $ unless $h_1=h_2$. Thus $H\h (G,X)$.
\end{enumerate}
\end{ex}

We will need a result from \cite{DGO}, which can be thought as a generalization of the first example above.

\begin{lem}[{\cite[Proposition 2.10]{DGO}}]\label{maln}
Let $G$ be a group, $H$ a hyperbolically embedded subgroup of $G$. Then $H$ is almost malnormal, i.e. $H\cap H^g$ is finite for every $g\in G\setminus H$.
\end{lem}

The following proposition relates the notions of a hyperbolically embedded subgroup and a relatively hyperbolic group.

\begin{prop}[{\cite[Proposition 4.28]{DGO}}]
Let $G$ be a group, $\Hl$ a finite collection of subgroups of $G$. Then $G$ is hyperbolic relative to $\Hl$ if and only if $\Hl\h (G,X)$ for some finite subset $X\subseteq G$.
\end{prop}


\section{Classification of acylindrical actions on hyperbolic spaces}\label{Secclass}


We recall the standard classification of groups acting on hyperbolic spaces, which goes back to Gromov \cite[Section 8.2]{Gro}. Let $G$ be a group acting on a hyperbolic metric space $S$. By $\Lambda (G)$ we denote the set of limit points of $G$ on $\partial S$. That is, $\Lambda (G)$ is the set of accumulation points of any orbit of $G$ on $\partial S$. Possible actions of groups on hyperbolic spaces break in the following $4$ classes according to $|\Lambda (G)|$.

\begin{enumerate}
\item[1)] $|\Lambda (G)|=0$. Equivalently,  $G$ has bounded orbits. In this case $G$ is called \emph{elliptic}.

\item[2)] $|\Lambda (G)|=1$. Equivalently, $G$ has unbounded orbits and contains no loxodromic elements. In this case $G$ is called \emph{parabolic}.

\item[3)] $|\Lambda (G)|=2$. Equivalently, $G$ contains a loxodromic element and any two loxodromic elements have the same limit points on $\partial S$.

\item[4)] $|\Lambda (G)|=\infty$. Then $G$ always contains loxodromic elements. In turn, this case breaks into two subcases.
\begin{enumerate}
\item[a)] Any two loxodromic elements of $G$ have a common limit point on the boundary. In this case $G$ is called \emph{quasi-parabolic}.
\item[b)] $G$ contains infinitely many independent loxodromic elements.
\end{enumerate}
\end{enumerate}

The action of $G$ is called \emph{elementary}  in cases 1)--3) and non-elementary in case 4).
Although many proofs of the above classification assume properness of $X$, it also holds in the general case (see, e.g., the arguments in Sections 8.1-8.2 in \cite{Gro} or \cite{H} for complete proofs in a more general context).

The main goal of this section is to prove Theorem \ref{class}. Compared to the general classification, theorem rules out the parabolic and quasi-parabolic cases and characterizes case 3) in algebraic terms.  Before proceeding with the proof, we consider three examples showing that the theorem can fail in various ways even for (non-acylindrical) proper or free actions.

Note first that proper actions may not be acylindrical. Moreover every countable group $G$ acts properly on a locally finite hyperbolic graph. This was noticed by Gromov \cite{Gro} and can be proved by attaching the Groves--Manning combinatorial horoball to a Cayley graph $\Gamma=\Gamma(G,X)$ with respect to a finite generating set $X$. We explain the argument in the case  when $G$ is finitely generated and refer to \cite{Hru} for the general case.

Consider the locally finite graph $\mc{H}(G)$ constructed as follows. The vertex set of  $\mc{H}(G)$ is $G\times \left( \{0\}\cup \mathbb N \right)$. The edge set of $\mc{H}(G)$ contains the following two types of edges:
\begin{enumerate}
\item[(a)] For every $k\ge 0$ and every $u,v\in G$ such that $0<\d _X(u,v)\leq 2^k$, there is an edge
  connecting $(u,k)$ to $(v,k)$.
\item[(b)]  For every $k\ge 0$ and $v\in G$, there is an edge  joining   $(v,k)$ to $(v,k+1)$.
\end{enumerate}

Clearly the map $G\to G\times \{ 0\}$ extends to an injection $\Gamma \to \mc{H}(G)$ and the action of $G$ on itself naturally extends to a parabolic action on $\mc{H}(G)$. It is not hard to show (see \cite{GM}) that the graph $\mc{H}(G)$ is always hyperbolic. Clearly the action of $G$ on $\mc{H}(G)$ is proper since $X$ is finite, but not acylindrical unless $G$ is finite. This leads to the following example.

\begin{ex}\label{ex1}
There exists a group $G$ acting properly on a locally finite hyperbolic graph such that every element of $G$ is elliptic, but orbits of $G$ are unbounded. Indeed let $G$ be a finitely generated infinite torsion group. Then the action of $G$ on $\mc{H}(G)$ satisfies all requirements. This never happens for acylindrical actions since, by Theorem \ref{class}, if every element of $G$ has bounded orbits, then $G$ has bounded orbits.
\end{ex}

\begin{ex}\label{ex2}
The action of the Baumslag-Solitar group $$BS(1,2)=\langle a,t \mid t^{-1}at=a^2\rangle $$ on its Bass-Serre tree provides an example of a quasi-parabolic action. A quasi-parabolic free action of $BS(1,2)$ on $\mathbb H^2$ can be obtained via the standard embedding $BS(1,2)\to SL_2(\mathbb R)$. Obviously the action of $BS(1,2)$ is non-elementary in both cases, but every two loxodromic elements have a common limit point.
\end{ex}

\begin{ex}\label{ex3}
Finally we note that a group $G$ acting elementary on a hyperbolic space and containing a loxodromic element can be arbitrary large (for example, consider the action of $\mathbb Z \times H$ on the line induced by the first factor, where $H$ is an arbitrary group; by replacing the line with a cylinder with base a complete graph on $H$ this action can be made free).
\end{ex}

As we explained in the introduction, our contribution to Theorem \ref{class} amounts to ruling out parabolic actions. We begin with a particular case. We say that an action of a group $G$ on a metric space $S$ is \emph{uniformly proper} if for every $\e>0$, there exists $N>0$ such that for every $s\in S$, we have
$$
|\{ g\in G\mid \d (s,gs)\le \e\}|\le N.
$$
Obviously uniform properness implies acylindricity.

The following result should be compared to Example \ref {ex1}, which shows that proper actions on hyperbolic spaces can be parabolic. It is actually due to Ivanov and Olshanskii \cite{IO}. Although it is not stated explicitly, it can be easily extracted from the proof of Lemma 17 in \cite{IO}. Our language here is slightly different, so we provide the proof for completeness.

\begin{lem}[Ivanov--Olshanskii]\label{IO}
A uniformly proper action on a hyperbolic space cannot be parabolic. That is, if $G$ is a group acting on a hyperbolic space uniformly properly with unbounded orbits, then $G$ contains a loxodromic element.
\end{lem}

\begin{proof}
Let $G$ act on a $\delta$-hyperbolic space $S$. Without loss of generality we can assume that $\delta \ge 1$. Let $x$ be any point of $S$. We first note that if for some $g\in G$, we have $2(g^{-1}x, gx)_x\le  \d(x,gx) - 19\delta $, then the sequence of points $x, gx, g^2x, \ldots $ satisfies the assumptions of Lemma \ref{GH} (see Remark \ref{dsieq}) and hence $g$ is loxodromic. Thus we can assume that
\begin{equation}\label{dgxgx}
2(gx, g^{-1}x)_x>  \d(x,gx) - 19\delta
\end{equation}
for every $g\in G$.

Further suppose that for some $g,h\in G$ we have
\begin{equation}\label{dgxhx}
2(gx, hx)_x < a - 51\delta ,
\end{equation}
where $a=\min\{ \d(x,gx), \d(x,hx)\}$. Applying (\ref{Gprod}) twice and then (\ref{dgxgx}), we obtain
$$
\begin{array}{rcl}
2(gx, hx)_x & \ge & 2\min \{(gx, g^{-1}x)_x, (g^{-1}x,h^{-1}x )_x,(h^{-1}x, hx)_x\} - 32\delta \\&&\\
&\ge & \min\{ a-19\delta , 2(g^{-1}x, h^{-1}x)_x\} -32\delta.
\end{array}
$$
Together with (\ref{dgxhx}) this implies
\begin{equation}\label{dgh-1}
2(g^{-1}x, h^{-1}x)_x \le a - 19\delta .
\end{equation}
Inequalities (\ref{dgxhx}), (\ref{dgh-1}), and Remark \ref{dsieq} allow us to apply Lemma \ref{GH} again to the sequence of points $x,\; g^{-1}x,\; g^{-1}hx,\; g^{-1}hg^{-1}x,\; (g^{-1}h)^2x,\; \ldots $.  We conclude that $g^{-1}h$ is loxodromic. Thus it remains to consider the case when
\begin{equation}\label{gxhx-ge}
2(gx, hx)_x \ge \min\{ \d(x,gx), \d(x,hx)\} - 51\delta \;\;\;\;\; \forall\,g, h\in G.
\end{equation}
We will complete the proof by showing that (\ref{gxhx-ge})  leads to a contradiction.

\begin{figure}
  \centering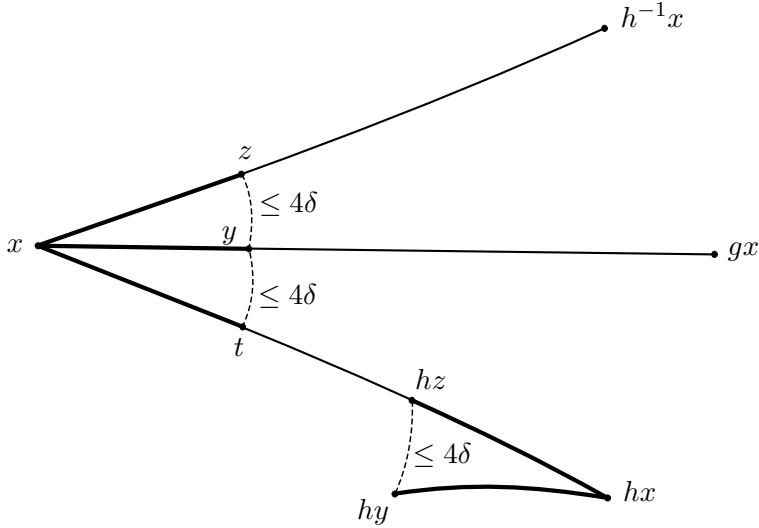
  \caption{Bold segments have length $5\delta$.}\label{fig03}
\end{figure}

Given $s\in S$, let $$A(s)=\{ g\in G\mid \d (s,gs)\le 71\delta \}.$$ By uniform properness of the action, there exists $N=\max _{s\in S} |A(s)|$. From now on, let $x$ denote a point of $S$ such that $|A(x)|=N$.

Since orbits of $G$ are unbounded, there exists $g\in G$ such that $g\notin A(x)$. That is,
\begin{equation}\label{dxgx}
\d(x,gx)> 71\delta .
\end{equation}
Let $y$ be the point on a geodesic segment $[x, gx]$ such that $\d(x,y)=5\delta$.  Consider any  $h\in G$ (possibly $h=g$) satisfying
\begin{equation}\label{dxhx}
\d(x,hx)\ge 61\delta
\end{equation}
and let $z$ (respectively, $t$) be the point on a geodesic segment $[x, h^{-1}x]$ (respectively, $[x, hx]$) such that $\d(x,z)=5\delta$ (respectively, $\d(x, t)=5\delta$, see Fig. \ref{fig03}). Since the points $x$, $t$, $hz$, and $hx$ belong to the same geodesic $[x, hx]= h [h^{-1}x,x]$, we obtain
\begin{equation}\label{txz}
\d (t, hz)= \d(x, hx) - \d (x, t) - \d(hx, hz)=\d(x, hx) - 10\delta.
\end{equation}
By (\ref{gxhx-ge}), (\ref{dxgx}), and (\ref{dxhx}) we have $(gx, h^{\pm 1}x)_x\ge 5\delta$. Hence by part (b) of Lemma \ref{GHs23}, we obtain $\d (y,z)\le 4\delta$ and $\d(y,t)\le 4\delta$. Combining this with (\ref{txz}), we obtain
\begin{equation}\label{yhy1}
\d (y, hy) \le \d(y,t) +\d (t, hz) +\d (hz, hy)\le \d(x, hx)-2\delta.
\end{equation}

Inequality (\ref{yhy1}) implies that if $h\in A(x)$ and satisfies  (\ref{dxhx}), then $h\in A (y)$.
On the other hand, if $h$ does not satisfy (\ref{dxhx}), then  we  have
$$
\d (y, hy) \le \d(y,x) +\d (x, hx) +\d (hx,hy)\le \d(x, hx) +10 \delta \le 71\delta
$$
and hence $h\in A(y)$ again. Thus $A(x)\subseteq A(y)$. In addition, inequality (\ref{yhy1}) applied to $g=h$ yields $\d (y,gy)\le \d(x, gx)-2\delta $. We now set $y_0=y$ and iterate the process, i.e., we use $y_0=y$ in place of $x$ to construct $y_1$, then use $y_1$ in place of $x$ to construct $y_2$, etc. Repeating this procedure sufficiently many times, we can ensure that $A(x)\cup \{ g\} \subseteq A(y_n)$ for some $n$. However this implies $|A(y_n)| >|A(x)|$, which contradicts the choice of $x$.
\end{proof}

To prove Theorem \ref{class}, we will need two more lemmas.

It is well-known that if a group $G$ acts on a tree and $g,h$ are two elliptic elements of $G$ with $Fix(g)\cap Fix(h)=\emptyset$, then $gh$ is loxodromic. The next lemma may be thought of as a hyperbolic analogue of this fact obtained by ``quasification"; the elements $x$ and $y$ play the role of the (quasi) fixed sets, see (\ref{max}), while the inequality (\ref{min}) is the counterpart of the disjointness condition.

\begin{lem}\label{prel}
Let $G$ be a group acting on a hyperbolic space $(S,d)$. Suppose that for some $g,h\in G$ there exist $x,y\in S$ and $C>0$ such that
\begin{equation}\label{max}
\max\{ \d (x,gx), \d(y,hy)\} \le C
\end{equation}
and
\begin{equation}\label{min}
\min \{ \d (x,hx) , \d (y, gy)\} \ge \d (x,y) + 2C +18\delta+1.
\end{equation}
Then $gh$ is a loxodromic isometry.
\end{lem}

\begin{proof}
Consider the sequence of points $x_0, y_0, x_1, y_1,  \ldots $,  where $x_i=(gh)^i x$ and $y_i=(gh)^igy$. We will show that
\begin{equation}\label{dx0xi}
\d(x, (gh)^ix)=\d (x_0, x_i)\ge i
\end{equation}
for every $i\in \mathbb N$ by applying Lemma \ref{GH}. Obviously (\ref{dx0xi}) implies that $gh$ is loxodromic.

To apply Lemma \ref{GH} we have to verify that
\begin{equation}\label{GH1}
\d (x_{i}, x_{i+1})\ge \max \{ \d (x_{i}, y_i), \d (x_{i+1}, y_i)\} +18\delta +1
\end{equation}
and
\begin{equation}\label{GH2}
\d (y_{i}, y_{i+1})\ge \max \{ \d (y_{i}, x_{i+1}), \d (y_{i+1}, x_{i+1})\} +18\delta +1
\end{equation}
for all $i\ge 0$. Using (\ref{max}) and (\ref{min}), we obtain
\begin{equation}\label{GH11}
\d (x_{i}, x_{i+1}) = \d (x, ghx)
 \ge  \d (gx, ghx)-\d (gx, x)  =\d (x,hx)-\d(gx,x)\ge \d (x,y) + C +18\delta +1.
\end{equation}
On the other hand,
\begin{equation}\label{GH12}
\d (x_{i}, y_i) =  \d (x, gy) \le \d (x, gx)+\d (gx, gy)
 \le   C +\d (x,y).
\end{equation}
and similarly
\begin{equation}\label{GH13}
\d (x_{i+1}, y_i) =\d (hx,y)=\d (x, h^{-1}y) \le \d(x,y) +\d (y, h^{-1}y) \le \d(x,y) +C.
\end{equation}

Obviously (\ref{GH11})--(\ref{GH13}) imply (\ref{GH1}). The proof of (\ref{GH2}) is analogous.
\end{proof}

\begin{lem}\label{strongacyl}
Let $G$ be a group acting acylindrically on a hyperbolic space $(S, \d)$. Then for any $\e >0$ there exist $R, N>0$ such that the following holds. For any $x,y\in S$ such that $\d (x,y)\ge R$, the set of elements $g\in G$ satisfying
\begin{equation}\label{gxgy}
\d (x, gx)\le \e \;\;\; {\rm and} \;\;\; \d (y, gy)\le \d (x,y) +\e
\end{equation}
has cardinality at most $N$.
\end{lem}

\begin{proof}
Fix $\e>0$. Let $S$ be $\delta$-hyperbolic. By acylindricity there exist $R_0, N$ such that for every $u,v\in S$ with $\d (u,v)\ge R_0$, we have
\begin{equation}\label{uv}
| \{ g\in G\mid \max\{ \d(u, gu), \d(v,gv)\} <16\delta + \e \} | \le N.
\end{equation}
Without loss of generality we can assume that
\begin{equation}\label{Red}
R_0\ge 2\e .
\end{equation}
Let $R=3R_0$ and let $x,y\in S$ satisfy
\begin{equation}\label{dxy}
\d (x,y)\ge R\ge 3R_0.
\end{equation}
Let $p$ be a geodesic in $S$ connecting $x$ to $y$, and let $m$ be the point on $p$ such that $\d(x,m)=R_0$ (see Fig. \ref{fig02}). Further let   $g\in G$ be any element satisfying (\ref{gxgy}). It suffices to show that
\begin{equation}\label{gm}
\d (m, gm)\le 16\delta +\e.
\end{equation}
Indeed then (\ref{uv}) applied to $u=x$ and $v=m$ completes the proof.

\begin{figure}
  \centering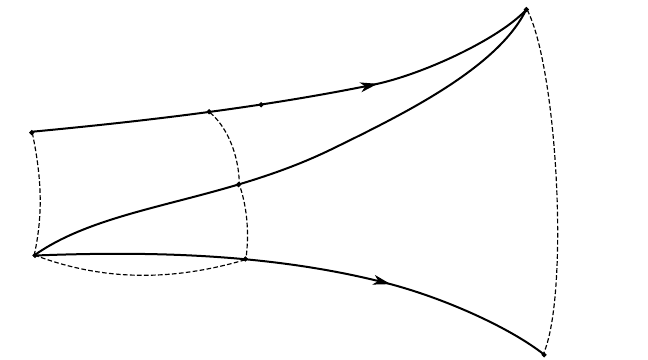
  \caption{}\label{fig02}
\end{figure}

Note first that
\begin{equation}\label{abs}
|\d (x, gy) -\d(x,y)|= |\d (x,gy)-\d (gx,gy)|\le \d (x,gx)\le \e.
\end{equation}
Using this and the the second inequality in (\ref{gxgy}) we obtain $(y, gy)_x \ge \frac12 \d(x,y) - \e$. Further applying (\ref{dxy}) and (\ref{Red}) we obtain $(y, gy)_x \ge \frac32R_0 - \e\ge R_0$. Let $n$ be the point on $[x, gy]$ such that $\d(x,n)=\d(x,m)=R_0$. Then by Lemma \ref{GHs23} (b) we obtain $\d (m,n)\le 4\delta $.

Further let $k$ be the point on $gp$ such that $\d (gy, k)=\d(gy, n)$. As in the previous paragraph, using (\ref{abs}) and then (\ref{dxy}) and (\ref{Red}) we obtain $$(x,gx)_{gy} \ge \d (x,gy)-\e > \d(x,gy) - R_0= \d (gy, n)$$ and hence $\d(n,k)\le 4\delta $.

Thus
\begin{equation}\label{dmk}
\d (m,k)\le \d(m,n)+\d(n,k)\le 8\delta .
\end{equation}
Since $gx$, $gm$, and $k$ belong to the same geodesic, we have
$$
\d(gm, k)\le |\d (gx, gm)-\d (gx, k)|= |\d (x, m)-\d (gx, k)| \le \d (x,gx) +\d (m,k)\le \e+8\delta.
$$
Finally, from the last inequality and (\ref{dmk}) we obtain (\ref{gm}).
\end{proof}

\begin{proof}[Proof of Theorem \ref{class}]
Let us first prove that if $G$ has unbounded orbits, then $G$ contains loxodromic elements. If the action of $G$ is uniformly proper, this follows from Lemma \ref{IO}. Hence we can assume that there exists $C>0$ such that quasi-stabilizers $$A(x)=\{ g\in G\mid \d (x,gx)\le C\}$$ can be arbitrary large. Let $\e = C +18\delta +1$, and let $R=R(\e)$, $N=N(\e )$ be the constants provided by Lemma \ref{strongacyl}. Let $x\in S$ be a point such that $|A(x)|\ge  2N+1$.

Since orbits of $G$ are unbounded, there exists $t\in G$ such that $\d (x, tx) \ge R$. Let $y=tx$. Obviously for every $a\in A(x)$, the element $tat^{-1}$ satisfies
$$
\d (y, tat^{-1}y)=\d (x, ax)\le C<\e.
$$
Let $B_1$ (respectively, $B_2$) be the subset of $A(x)$ consisting of all $a\in A(x)$ such that $\d (y, ay) < \d (x,y) +\e$ (respectively, $\d (x, tat^{-1}x)< \d(x,y) +\e$). By Lemma \ref{strongacyl}, $|B_1|\le N$ and $|B_2|\le N$. Since $|A(x)|\ge  2N+1$, there exists $a\in A(x)$ such that $$\d (y, ay) \ge  \d (x,y) +\e= \d(x, y) +C+18\delta +1$$ and $$\d (x, tat^{-1}x)\ge  \d(x,y) +\e= \d(x, y) +C+18\delta +1.$$ Now we can apply Lemma \ref{prel} to elements $g=a$ and $h=tat^{-1}$. We obtain that $gh$ is loxodromic.

Thus if no element of $G$ is loxodromic, $G$ is of type (a). Suppose now that $G$ contains a loxodromic element $g$. Then it is either of type (b) or of type (c) by Lemma \ref{wpd-el}.
\end{proof}


\section{Separating cosets of hyperbolically embedded subgroups}


Our next goal is to review and generalize some technical tools from \cite{HO}, which will be used in the proof of Theorem \ref{main}. Throughout this section, we use the notation introduced in Section 2.4. We fix a group $G$ with a hyperbolically embedded collection of subgroups $\Hl\h (G,X)$.

All results discussed in this section can be easily extracted from \cite[Section 3]{HO}, although our settings are slightly different there. In particular, in \cite{HO} only separating cosets of a single subgroup $H_\lambda $ were considered. However in the proof of Theorem \ref{main} we will have to deal with separating cosets of all subgroups from the collection $\Hl$ simultaneously. Another (rather minor) difference occurs in the definition of a separating coset, see Remark \ref{remscd}. Since results of this section are crucial for our paper, we write down complete proofs whenever proofs from \cite{HO} do not work verbatim.

\begin{defn}\label{comp}
Let $q$ be a path in the Cayley graph $\G $. A (non-trivial) subpath $p$ of $q$ is called an \emph{$H_\lambda $-subpath}, if the label of $p$ is a word in the alphabet $H_\lambda   $. An $H_\lambda$-subpath $p$ of $q$ is an {\it $H_\lambda $-component} if $p$ is not contained in a longer $H_\lambda$-subpath of $q$; if $q$ is a loop, we require in addition that $p$ is not contained in any longer $H_\lambda $-subpath of a cyclic shift of $q$. Further by a {\it component} of $q$ we mean an $H_\lambda $-component of $q$ for some $\lambda \in \Lambda$.

Two $H_\lambda $-components $p_1, p_2$ of a path $q$ in $\G $ are called {\it connected} if there exists a
path $c$ in $\G $ that connects some vertex of $p_1$ to some vertex of $p_2$, and ${\Lab (c)}$ is a word consisting only of letters from $H_\lambda   $. In algebraic terms this means that all vertices of $p_1$ and $p_2$ belong to the same left coset of $H_\lambda $. Note also that we can always assume that $c$ is an edge as every element of $H_\lambda $ is included in the set of generators. A component of a path $p$ is called \emph{isolated} in $p$ if it is not connected to any other component of $p$.

It is convenient to extend the relative metric $\dl $ defined in Section 2.4 to the whole group $G$ by assuming $$\dl (f,g)\colon =\left\{\begin{array}{ll}\dl (f^{-1}g,1),& {\rm if}\;  f^{-1}g\in H_\lambda \\ \dl (f,g)=\infty , &{\rm  otherwise.}\end{array}\right.$$
\end{defn}

The following result is a simplified version of \cite[Proposition 4.13]{DGO}. Recall that a path $p$ in a metric space is $(\mu, b)$-quasi-geodesic for some $\mu \ge 1$ and $b\ge 0$ if for every subpath $q$ of $p$ one has $$ \ell(q)\le \mu \d (q_-, q_+)+b.$$

\begin{lem}\label{C}
For every $\mu\ge 1$ and $b\ge 0$, there exists a constant $C=C(\mu,b)>0$ such that for any $n$-gon $p$ with $(\mu, b)$-quasi-geodesic sides in $\G$, any $\lambda \in \Lambda$, and any isolated $H_\lambda$-component $a$ of $p$, we have $\dl (a_-, a_+)\le Cn$.
\end{lem}

From now on, we fix a constant $C=C(1,0)$ provided by Lemma \ref{C}. We also fix any constant
\begin{equation}\label{DC}
D\ge 3C.
\end{equation}
In fact, for the purpose of proving Theorem \ref{main} we can just take $D=3C$ as in \cite{HO}. However the more general assumption (\ref{DC}) can be useful for future applications (see for example Proposition \ref{bgen}) and does not cause any complications.

\begin{defn}\label{sepcosdef}
We say that a path $p$ in $\G $ penetrates a coset $xH_\lambda$ for some $\lambda\in \Lambda $ if $p$ decomposes as $p_1ap_2$, where $p_1, p_2$ are possibly trivial, $(a)_-\in xH_\lambda$, and $a$ is an $H_\lambda$-component of $p$. Note that if $p$ is geodesic, it penetrates every coset of $H_\lambda$ at most once. In this case we call $a$ the \emph{component of $p$ corresponding to $xH_\lambda$}. If, in addition, $\dl (a_-, a_+)> D$, then we say that $p$ \emph{essentially penetrates} the coset $xH_\lambda $

If some $f,g\in G$ are joined by a geodesic path $p$ in $\G$  that essentially penetrates $xH_\lambda $, we call $xH_\lambda $ an \emph{$(f,g;D)$-separating coset}.
The set of all $(f,g;D)$-separating cosets of subgroups from the collection $\Hl$ is denoted by $S(f,g;D)$. Note that $S(f,g;D)=S(g,f;D)$
\end{defn}

\begin{rem}\label{remscd}
If $D=3C$, then our notion of an $(f,g;D)$-separating coset almost coincides with the notion of an $(f,g)$-separating coset from \cite{HO} (see Definition 3.1 there). The difference is that a coset $xH_\lambda$ was also called separating in \cite{HO} whenever $f,g\in xH_\lambda $ and $f\ne g$. This variant of the definition was justified by some technical reasons specific to \cite{HO} and would cause unnecessary complications in this paper.
\end{rem}

The lemma below follows immediately from Lemma \ref{C}.

\begin{lem}[cf. {\cite[Lemma 3.3 (a)]{HO}}]\label{sep}
For any $f,g\in G$ and any $xH_\lambda\in S(f,g;D)$, every path in $\G$ connecting $f$ to $g$ and composed of at most $2$ geodesic segments penetrates $xH_\lambda$. Moreover, if $D\ge 3C(\mu,b)$ for some $\mu\ge 1$, $b\ge 0$, where $C(\mu,b)$ is the constant provided by Lemma \ref{C}, then every path in $\G$ connecting $f$ to $g$ and composed of at most $2$ $(\mu,b)$-quasi-geodesic segments penetrates $xH_\lambda$.
\end{lem}

The next result is also fairly elementary.

\begin{lem}[{\cite[Lemma 3.5]{HO}}] \label{dist}
Suppose that a geodesic $p$ in $\G$ penetrates a coset $xH_\lambda$ for some $\lambda \in \Lambda$. Let $p=p_1ap_2$, where $a$ is a component of $p$ corresponding to $xH_\lambda$. Then $\ell (p_1)=\dxh (p_-, xH_\lambda)$.
\end{lem}

We now state the analogue of Definition 3.6 from \cite{HO}.

\begin{defn}\label{order}.
Given any $f,g\in G$, we define a relation $\preceq$ on the set $S(f,g;D)$ as follows: for any $C_1,C_2\in S(f,g;D)$,
$$C_1 \preceq C_2\;\;\; {\rm  iff}\;\;\; \dxh (f, C_1 )\le \dxh (f, C_2 ).$$
\end{defn}

As in \cite{HO}, from Lemma \ref{sep} and Lemma \ref{dist} we immediately obtain the following.

\begin{lem}\label{sep-ref}
For any $f,g\in G$, $\preceq$ is a linear order on $S(f,g;D)$ and every geodesic $p$ in $\G$ going from $f$ to $g$ penetrates all $(f,g;D)$-separating cosets according to the order $\preceq$. That is, $$S(f,g; D)=\{ C_1\preceq C_2\preceq \ldots \preceq C_n\}$$ for some $n\in \mathbb N$ and $p$ decomposes as $$p=p_1a_1\cdots p_na_np_{n+1},$$ where $a_i$ is the component of $p$ corresponding to $C_i$, $i=1, \ldots , n$ (see Fig. \ref{fig1}).
\end{lem}
\begin{figure}
 \centering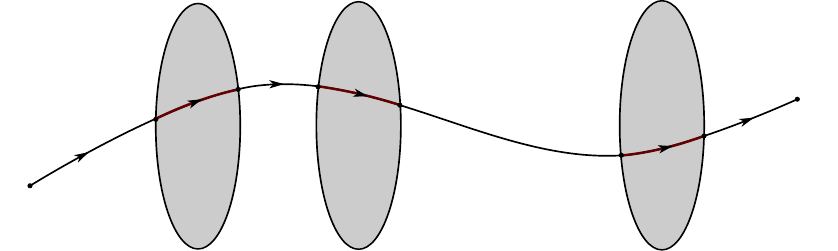\\
  \caption{}\label{fig1}
\end{figure}

The next result is a generalization of a simplification of \cite[Lemma 3.9]{HO}.

\begin{lem}\label{HO}
For any $f,g,h\in G$, the set $S(f,g;D)$ can be decomposed as $$S(f,g;D)= S^\prime\sqcup S^{\prime\prime}\sqcup F,$$ where $S^\prime\subseteq S(f,h;D)$, $S^{\prime\prime}\subseteq S(g,h; D)$, and $|F|\le 2$. In particular, we have
$$
\begin{array}{rcl}
|S(f,g; D)| &\le& |S(f,h; D) | + |S(g,h; D)| + 2 .
\end{array}
$$
\end{lem}

\begin{proof}
If $|S (f,g; D)|\leq 2$ the statement is trivial, so we can assume $|S (f,g; D)|>2$. We fix any geodesics $q$ and $r$ in $\G$ connecting  $h$ to $g$ and $f$ to $h$, respectively. By Lemma \ref{sep}, every coset from $S(f,g;D)$ is penetrated by at least one of $q$, $r$. Without loss of generality we may assume that at least one of the cosets from $S(f,g;D)$ is penetrated by $r$. Let $$S(f,g;D)=\{ C_1\preceq C_2\preceq \ldots \preceq C_n\}$$ and let $C_i$ be the largest coset (with respect to the order $\preceq$) that is penetrated by $r$.

Assume that $i\ge 2$. For every $1\le j< i $, there exists a geodesic $p$ connecting $f$ to $g$ in $\G$ such that $p$ essentially penetrates $C_j$. Further by Lemma \ref{sep-ref}, $p$ decomposes as $$ p=p_1a_1p_2a_2p_3,$$ where $a_1$, $a_2$ are components of $p$ corresponding to $C_j$ and $C_i$, respectively. Similarly $r$ decomposes as $$r=r_1br_2,$$ where $b$ is a component of $r$ corresponding to $C_i$ (see Fig. \ref{fig2}).

Since $(r_2)_-$ and $(p_2)_+$ belong to the same coset $C_i$ of $H_\lambda$, there exists an edge $e$ in $\G$ going from $(p_2)_+$ to $(r_2)_-$. By Lemma \ref{dist}, we have $\ell (p_1a_1p_2)=\ell (r_1)$.  Hence for the path $t=p_1a_1p_2er_2$, we have
 $$
 \ell(t)= \ell(p_1a_1p_2) +1+\ell(r_2)=\ell (r_1)+1+\ell(r_2)=\ell(r).
 $$
Thus $t$ is also geodesic. Since $p$ essentially penetrates $C_j$, so does $t$. Therefore, $C_j\in S (f,h; D)$ for every $1\le j< i $.

By the choice of $i$, if $i<n$ then the coset $C_{i+1}$ is penetrated by $q$. Arguing as above we obtain that $C_j\in S (g,h; D)$ for every $i+1<j\le n$. It remains to define $$S^\prime = \{ C_j\mid 1\le j<i\},\;\;\; S^{\prime\prime} = \{ C_j\mid i+1<j\le n\},\;\;\; F=S(f,g; D)\setminus (S^\prime \cup S^{\prime\prime}).$$
\end{proof}

\begin{figure}
 \centering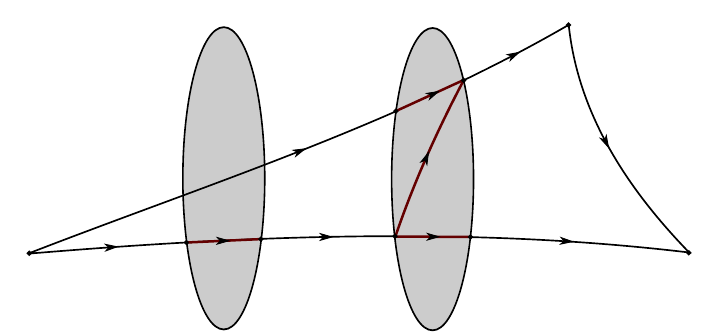\\
  \caption{}\label{fig2}
\end{figure}


\section{Proof of Theorem \ref{main}}


In this section we prove Theorem \ref{main}. As we already mentioned in the introduction, we only need to show that (AH$_4$) implies (AH$_1$). In what follows we say that a Cayley graph $\Gamma (G,X)$ is \emph{acylindrical} if the action of $G$ on $\Gamma (G,X)$  is acylindrical. For Cayley graphs, the acylindricity condition can be rewritten as follows: for every $\e>0$ there exist $R>0$ and $N>0$ such that for any $g\in G$ of length $|g|_X\ge R$ we have $$ |\{ f\in G\mid |f|_X\le \e , \; |g^{-1}fg|_X \le \e \} | \le N.$$
Note that it suffices to verify this condition for all sufficiently large $\e $.

The next lemma is obvious.

\begin{lem}\label{acyl}
For any group $G$ and any generating sets $X$ and $Y$ of $G$ such that $$\sup_{x\in X} |x|_Y<\infty\;\;\; {\rm and}\;\;\; \sup_{y\in Y} |y|_X <\infty ,$$ the following hold.
\begin{enumerate}
\item[(a)] $\Gamma (G,X)$ is hyperbolic if and only if $\Gamma (G,Y)$ is hyperbolic.
\item[(b)] $\Gamma (G,X)$ is acylindrical if and only if $\Gamma (G,Y)$ is acylindrical.
\end{enumerate}
\end{lem}

We begin by discussing the case of relatively hyperbolic groups. To the best of our knowledge, the proposition below has never been recorded before.

\begin{prop} \label{rhg}
Let $G$ be a group hyperbolic relative to a collection of subgroups $\Hl$. Let $X$ be a finite relative generating set of $G$ with respect to $\Hl$. Then the Cayley graph $\G$ is acylindrical.
\end{prop}

Our proof is based on the fact that for a suitable $X$, geodesic triangles in $\G$ satisfy the Rips condition with respect to the \emph{locally finite} word metric $\d _X$. In case $G$ is finitely generated, this is \cite[Theorem 3.26]{Osi06}. Note however that the proof from \cite{Osi06} works in the case when $G$ is infinitely generated equally well if we allow word metrics to take infinite values, i.e., if we consider word metrics (and Cayley graphs) with respect to non-generating sets as in this paper.

\begin{lem}\label{rips}
Let $G$ be a group hyperbolic relative to a collection of subgroups $\Hl$. Then there exists $\xi>0$ and a finite relative generating set $X$ of $G$ with respect to $\Hl$ such that the following condition holds. For every triangle with geodesic sides $p$, $q$, $r$ in $\G$ and every vertex $u$ on $p$, there is a vertex $v$ on $q\cup r$ such that $\d _X (u,v)\le \xi$.
\end{lem}

\begin{figure}
 \centering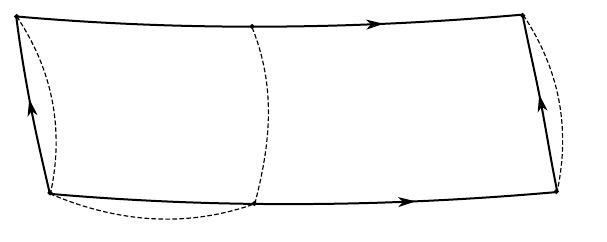\\
  \caption{}\label{fig07}
\end{figure}

\begin{proof}[Proof of Proposition \ref{rhg}]
By Lemma \ref{acyl} it suffices to prove the proposition for some finite relative generating set $X$. Let $X$ and $\xi$ be as in Lemma \ref{rips}. Let us fix $\e>0$. We can assume that $\e\in \mathbb N$ and $\e\ge \xi $.

Let $g\in G$ be an element of length $|g|_{X\cup \mathcal H} \ge 6\e+2$. Let $f$ be any element of $G$ such that
\begin{equation}\label{ff}
|f|_{X\cup \mathcal H}\le \e\;\;\;{\rm and}\;\;\; |g^{-1}fg|_{X\cup \mathcal H}\le \e.
\end{equation}
Denote by $p$ any geodesic in $\G$ going from $1$ to $g$ and let $q=fp$; thus $q$ is a geodesic going from $f$ to $fg$ and $\Lab(p)\equiv\Lab(q)$. Also let $s_1$ and $s_2$ be geodesics in $\G$ going from $1$ to $f$ and from $g$ to $fg$, respectively.

Let $u$ be a vertex of $p$ such that
\begin{equation}\label{d1u}
\dxh (1,u)=3\e +1.
\end{equation}
Note that we also have
\begin{equation}\label{dug}
\dxh (u,g)=|g|_{X\cup \mathcal H} - (3\e +1)\ge 3\e +1.
\end{equation}
Drawing a diagonal in the geodesic quadrangle $ps_2q^{-1}s_1^{-1}$ and applying Lemma \ref{rips} twice, we obtain that there exists a vertex $v$ on $q\cup s_1\cup s_2$ such that $\d_X (u,v)\le 2\xi\le 2\e$. Note that if $v\in s_1$,  then $$\dxh (u,1)\le \dxh (u,v) + \dxh(v,1) \le \d_X(u,v) +\ell(s_1)\le 3\e,$$ which contradicts (\ref{d1u}). Similarly using (\ref{dug}), we conclude that $v$ cannot belong to $s_2$. Thus $v\in q$.

We have $$f=u\cdot (u^{-1}v) \cdot (v^{-1}f).$$ Since $\Lab (q)\equiv \Lab (p)$, $f^{-1}v$ is represented by the label of the initial subpath of $q$ of length at most $$\dxh (f,v)\le \dxh (f,1)+\dxh (1,u)+\dxh(u,v)\le 6\e+1.$$ Recall that $\Lab (q)\equiv \Lab (p)$. Therefore, the element $v^{-1}f$ is represented by the label of an initial segment of $p$ of length at most $6\e +1$. Hence there are $6\e +2$ possible choices for this element. Recall also that $|u^{-1}v|_X\le 2\e$. Thus there are at most $(6\e+2) |B_X(2\e)|$ elements $f$ satisfying (\ref{ff}), where $B_X(2\e)$ is a ball of radius $2\e$ in $G$ with respect to the metric $\d _X$. Since $X$ is finite, so is $B_X(2\e)$.
\end{proof}

The analogue of Proposition \ref{rhg} for groups with hyperbolically embedded collections of subgroups does not hold in general. A simple counterexample is the group $G=(K\times  \mathbb Z)\ast H$, where $K$ is an  infinite group and $H$ is non-trivial. Let $X=K\cup \{x\}$, where $x$ is a generator of $\mathbb Z$. It is easy to verify that $H\h (G,X)$. However the action of $G$ on $\G $ is not acylindrical, as any element of $K$ moves any vertex of the infinite geodesic ray in $\G $ starting from $1$ and labelled by the infinite power of $x$ by a distance at most $1$. On the other hand, if we enlarge the relative generating set to $Y=K\times \mathbb Z$, then $\Gy$ is quasi-isometric to the Bass-Serre tree associated to the free product structure of $G$ and is obviously acylindrical. The main result of this section shows that a similar trick can be performed in the general case.

\begin{thm}\label{XtoY}
Let $G$ be a group, $\Hl $ a finite collection of subgroups of $G$, $X$ a subset of $G$. Suppose that $\Hl\h (G,X)$. Then there exists $Y\subseteq G$ such that $X\subseteq Y$ and the following conditions hold.
\begin{enumerate}
 \item[(a)] $\Hl\h (G,Y)$. In particular, the Cayley graph $\Gamma (G,Y\sqcup \mathcal H)$ is hyperbolic.
 \item[(b)] The action of $G$ on $\Gamma (G,Y\sqcup \mathcal H)$ is acylindrical.
\end{enumerate}
\end{thm}

The proof of Theorem \ref{XtoY} is given in a sequence of lemmas below. Throughout the rest of this section we are working under notation and assumptions of Theorem \ref{XtoY}. Recall that we have already fixed a constant $C>0$ satisfying Lemma \ref{C} and $D$  satisfying (\ref{DC}).
We will show that the subset $Y=Y(D)\subseteq G$ defined by
\begin{equation}\label{Y}
Y=\{ y\in G \mid S(1,y;D)=\emptyset\}
\end{equation}
satisfies conditions (a) and (b) from Theorem \ref{XtoY}.

To prove hyperbolicity of $\Gy$ we will make use of a sufficient condition from \cite{KR}. We state it in a much simplified form here. (In the notation of \cite{KR}, we have $L=M_1=1$ and $M_2=M$.)

\begin{lem}[{\cite[Corollary 2.4]{KR}}]\label{KR}
Let $\Sigma $ be a graph obtained from a graph $\Gamma $ by adding edges. Suppose that $\Gamma $ is hyperbolic and there exists $M>0$ such that for any two vertices $x,y$ of $\Gamma $ connected by an edge in $\Sigma$ and  any geodesic $p$ in $\Gamma $ going from $x$ to $y$, the diameter of $p$ in $\Sigma $ is at most $M$. Then $\Sigma $ is also hyperbolic.
\end{lem}

\begin{lem}\label{hypGy}
We have $X\subseteq Y $ and $\Gy$ is hyperbolic.
\end{lem}

\begin{proof}
If $x\in X$, $x\ne 1$, then the edge labelled by $x$ is a geodesic from $1$ to $x$ in $\G$. It does not contain any components and hence $S(1,x;D)=\emptyset$ by Lemma \ref{sep}. Thus $X\subseteq Y $. In particular, $\G$ is a subgraph of $\Gy$. According to Lemma \ref{KR}, it suffices to prove that for every $y\in Y $ and every geodesic $p$ in $\G$ with $p_-=1, p_+=y$, ${\rm diam}_{Y \cup\mathcal H} (p)$ is bounded from above by a uniform constant independent of $y$ and $p$.
Let $y$, $p$ be as above and let $v$ be any vertex on $p$. Let $p=p_1p_2$, where $(p_1)_+=(p_2)_-=v$. If $S(1,v; D)\ne \emptyset$, there exists a geodesic $q$ from $1$ to $v$ that essentially penetrates a coset $C\in S(1,v; D)$. Then $qp_2$ is a geodesic from $1$ to $y$ that essentially penetrates $C$. Thus $S(1,y;D)\ne \emptyset$, which contradicts (\ref{Y}). Therefore $S(1,v;D)= \emptyset$. This means that $v\in Y $ and hence $\dyh(1,v)\le 1$. Thus ${\rm diam}_{Y \cup\mathcal H} (p)\le 2$. It remains to apply Lemma \ref{KR}.
\end{proof}

To prove that $\Hl\h (G,Y )$ we will need the following result from \cite{DGO}.

\begin{lem}[{\cite[Corollary 4.27]{DGO}}]\label{X1X2}
Let $G$ be a group, $\Hl$ a collection of subgroups of $G$, $X_1,X_2\subseteq G$ two relative generating sets of $G$ with respect to $\Hl$. Suppose that $|X_1\vartriangle X_2|<\infty $. Then $\Hl\h (G,X_1)$ if and only if $\Hl \h (G, X_2)$.
\end{lem}

\begin{lem}\label{he}
$\Hl\h (G,Y )$.
\end{lem}
\begin{proof}
Since we already know that $\Gy$ is hyperbolic, it remains to verify the local finiteness condition (b) from Definition \ref{he-def}. We denote by $\dl^Y$ and $\dl^X$ the relative metrics on $H_\lambda $ defined using the graphs $\Gy$ and $\G$, respectively. In this notation, we have to verify that $(H_\lambda, \dl^Y)$ is locally finite for every $\lambda\in \Lambda$.

To simplify the proof, we will use the following trick. Let $$Z_\lambda =\{ h\in H_\lambda \mid \dl ^X(1,h)\le D\}$$ and
$$
Z=\left(\bigcup_{\lambda\in \Lambda}Z_\lambda \right) \cup X.
$$
Since every $(H_\lambda , \dl^X)$ is a locally finite space, every $Z_\lambda $ is finite. And since we assume that $|\Lambda |<\infty $, $|Z\vartriangle X|=|Z\setminus X|<\infty $. Hence $\Hl \h (G,Z) $ by Lemma \ref{X1X2}. We will denote by $\dl^Z$ the corresponding relative metric on $H_\lambda$.

Observe that the restriction of the natural inclusion $\G\to \Gz$ to the set of vertices is an isometry. Indeed it is clear from the definition of $Z$ that two vertices in $\Gz$ are connected by an edge if and only if they are connected by an edge in $\G$. However, some edges of $\G$ are doubled in $\Gz$. Indeed for every edge in $\G$ labelled by $h\in H_\lambda$ such that $\dl ^X(1,h)\le D$, there are two edges in $\Gz$ with the same endpoints: one labelled by $h\in H_\lambda \subseteq \mathcal H$ and the other labelled by the corresponding element of $Z$ (this is due to disjointness of the union $Z\sqcup \mathcal H$ in $\Gz$, see Remark \ref{disj}).

Below we will use the following construction. Let $e$ be an edge in $\Gy$ labelled by a letter from $Y$. Take a geodesic path $p_e$ in $\G$ connecting $e_-$ and $e_+$. Since $\Lab (e)\in Y$, we have $S(e_-, e_+;D)=\emptyset$. Hence $p_e$ does not essentially penetrate any coset of any $H_\lambda$ and hence every $H_\lambda$-component $c$ of $p_e$ satisfies $\dl ^X(c_-, c_+)\le D$. (Note that $c$ is a single edge as $p_e$ is geodesic.) Therefore, in $\Gz$ there is an edge labelled by a letter of $Z$ that connects $c_-$ to $c_+$. Thus passing to $\Gz$ and replacing all components of $p_e$ with the corresponding edges labelled by letters from $Z$, we obtain a path $q_e$ with the same endpoints as $e$ such that $\Lab (q_e)$ is a word in the alphabet $Z$. Since $\ell (q_e)=\ell(p_e)$ and inclusion $\G\to \Gz$ is isometric on the vertex set, $q_e$ is geodesic in $\Gz$.

Suppose now that an element $h\in H_\lambda $ satisfies $\dl ^Y(1,h)= n<\infty $ and let $r$ be the corresponding admissible path (see Definition \ref{dl}) of length $n$ in $\Gy $ that connects $1$ to $h$. Let $r=r_0s_1\ldots r_{m-1}s_mr_m$, where $s_1, \ldots, s_m$ are edges labelled by letters from $\mathcal H$ and $r_0, \ldots , r_m$ are paths (possibly trivial) labelled by words in the alphabet $Y$. Replacing each edge $e$ of each $r_i$ with the geodesic $q_e$ constructed in the previous paragraph, we obtain a path $t=t_0s_1\ldots t_{m-1}s_mt_m$ in $\Gz$ connecting $1$ to $h$, where $\Lab (t_i)$ is a word in the alphabet $Z$ for all $i=0, \ldots , m$. Note that $t$ is a concatenation of $n$ geodesics (in $\Gz$).

Let $f$ be the edge in $\Gz$ labelled by $h^{-1}\in H_\lambda $ and connecting $h$ to $1$. Obviously $s_1, \ldots , s_m$ are the only components of $t$. Since $r$ was admissible in $\Gy$, no $s_i$ contains edges of $\Gamma (H_\lambda, H_\lambda)$. Hence no $s_i$ can be connected to $f$ and thus $f$ is isolated in the loop $tf$.  Since $t$ consists of at most $n$ geodesic segments, from Lemma \ref{C} we obtain $$\dl^Z (1,h) \le C^Z(n+1)\le 2C^Z n \le 2C^Z \dl^Y (1,h),$$ where $C^Z$ is the corresponding constant for $\Gz$. Now local finiteness of $(H_\lambda, \dl ^Y)$ follows from that of $(H_\lambda, \dl ^Z)$.
\end{proof}

\begin{lem}\label{inY}
Let $f,g\in G$ and let $S(f,g;D)=\{C_1\preceq \ldots \preceq C_n\}$ for some $n\ge 1$. Then the following hold.
\begin{enumerate}
\item[(a)] There  exists $b\in C_1$ such that $f^{-1}b\in Y$.
\item[(b)] For every $1\le i<n$, there exist $a\in C_i$, $b\in C_{i+1}$ such that $a^{-1}b\in Y$.
\item[(c)] There  exists $a\in C_n$ such that $a^{-1}g\in Y$.
\end{enumerate}
\end{lem}

\begin{figure}
 \centering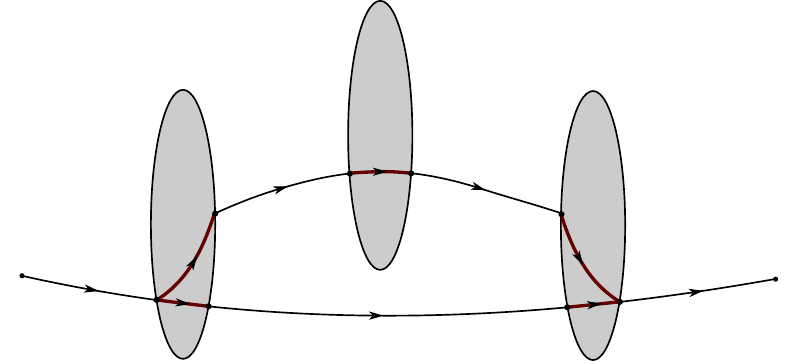\\
  \caption{}\label{fig06}
\end{figure}

\begin{proof}
We will only prove part (b); the proofs of (a) and (c) are the same (after obvious adjustments). Assume that $C_i$ and $C_{i+1}$ are cosets of $H_\lambda $ and $H_\mu$, respectively. We consider any elements $a\in C_i$, $b\in C_{i+1}$ such that
\begin{equation}\label{dab}
\dxh(a,b)=\dxh(C_i, C_{i+1}).
\end{equation}
Assume that $S(a,b;D)\ne \emptyset$. Let $p$ be a geodesic in $\G$ connecting $a$ to $b$ that essentially penetrates a coset $B\in S(a,b;D)$. Thus $p=p_1dp_2$, where $p_1, p_2$ are non-empty and $d$ is the component corresponding to $B$  (see Fig. \ref{fig06}). Consider any geodesic $q$ connecting $f$ to $g$ in $\G$. By Lemma \ref{sep-ref}, $q$ penetrates $C_i$ and $C_{i+1}$ and decomposes as $q=q_1e_1q_2e_2q_3$, where $e_1$, $e_2$ are components corresponding to $C_i$ and $C_{i+1}$. Let $f_1$ (respectively, $f_2$) be an edge connecting $(e_1)_-$ to $a$ (respectively, $b$ to $(e_2)_+$) and let $r=q_1f_1pf_2q_3$. By(\ref{dab}) we have $\ell (p)\le \ell (q_2)$. Hence $\ell (r)\le \ell (q)$. In particular, $r$ is a geodesic between $f$ and $g$ that essentially penetrates $B$. By Lemma \ref{sep-ref}, we have $B=C_i$ or $B=C_{i+1}$. If $B=C_i$, then $p_1$ is trivial and hence $p$ starts with an $H_\lambda $-component $d$. However this contradicts (\ref{dab}) since $d_+\in C_i$ in this case and $$\dxh (C_i, C_{i+1})\le\dxh (d_+, b)<\ell(p) =\dxh (a,b).$$ Similarly we get a contradiction if $B=C_{i+1}$.
\end{proof}

The next result relates the distance in $\Gy$ to the number of separating cosets between two points.

\begin{lem}\label{dS}
For any two elements $f,g\in G$, we have $$\frac12(\dyh (f,g)-1) \le |S(f,g;D)|\le 3\dyh (f,g) $$
\end{lem}

\begin{proof}
The left inequality immediately follows from Lemma \ref{inY}. The right inequality follows Lemma \ref{HO} by induction. Indeed if $\dyh(f,g)=0$, the statement is obvious. If $\dyh(f,g)>0$, then there exists $h\in G$ such that $\dyh(f,h)=\dyh(f,g)-1$ and either $h^{-1}g\in Y$ or $h^{-1}g\in \mathcal H$; in both cases we have $|S(h,g;D)|\le 1$. By induction and Lemma \ref{HO}, we obtain
$$
\begin{array}{rcl}
|S(f,g;D)|& \le & |S(f,h;D)|+|S(h,g;D)| +2 \\ && \\ & \le & 3\dyh(f,h) +3  \\ && \\ & = & 3(\dyh(f,g)-1)+3 \\ && \\ & = & 3\dyh (f,g).
\end{array}
$$
\end{proof}

\begin{figure}
 \centering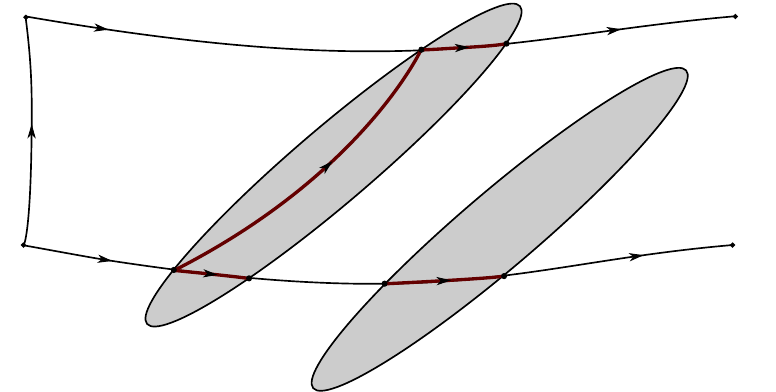\\
  \caption{}\label{fig04}
\end{figure}

\begin{lem}\label{acylGy}
$\Gy$ is acylindrical.
\end{lem}

\begin{proof}
We use the equivalent definition of acylindricity given by  Lemma \ref{=R}. Fix any $\e>0$; without loss of generality we can assume $\e\in \mathbb N$. Let $g$ be an element of $G$ such that
\begin{equation}\label{R}
|g|_{Y\cup\mathcal H}= 18\e +11.
\end{equation}
Suppose that
\begin{equation}\label{e-close}
\max\{ |f|_{Y\cup\mathcal H}, |g^{-1}fg|_{Y\cup\mathcal H}\} \le \e
\end{equation}
for some $f\in G$.

We first show that
\begin{equation}\label{acyl1}
|S(1,g;D)\cap S(f, fg;D)|\ge 3\e +1.
\end{equation}
Indeed suppose that $|S(1,g;D)\cap S(f, fg;D)|< 3\e +1$. Applying Lemma \ref{HO} twice we obtain a decomposition $S(1,g; D)= S_1\cup S_2\cup S_3\cup F$, where $S_1\subseteq S(1,f;D)$, $S_2\subseteq S(f, fg;D)$, $S_3\subseteq S(fg,g; D)$, and $|F|\le 4$. Using Lemma \ref{dS} we obtain
$$
\begin{array}{rcl}
|S(1,g;D)|& \le &  |S(1,f;D)|+|S(1,g;D)\cap S(f,fg;D)| + |S(g, fg;D)| +4 \\ &&\\
& < & 3\dyh (1,f) +3\e +1+ 3\dyh (fg, g) +4\\ &&\\
& \le & 9\e +5.
\end{array}
$$
Now using Lemma \ref{dS} again, we obtain $|g|_{Y\cup\mathcal H} \le 2|S(1,g;D)|+1< 18\e +11$, which contradicts (\ref{R}). Thus (\ref{acyl1}) holds.

Since $|S(1,f;D)|\le 3\dyh (1,f)\le 3\e$,  there exists a coset $C$ such that $$C\in S(1,g;D)\cap S(f, fg;D)\;\;\; {\rm and} \;\;\; C\notin S(1,f;D).$$ Let $p$ be a geodesic in $\G$ going from $1$ to $g$. Let also $q=fp$; thus $q$ is a geodesic in $\G$ connecting $f$ to $fg$. Since $C\in S(1,g;D)\cap S(f, fg;D)$, by Lemma \ref{sep-ref} $p$ and $q$ decompose as $p=p_1ap_2$ and $q=q_1bq_2$, where $a$ and $b$ are components corresponding to $C$ (Fig. \ref{fig04}). Let $u=a_-$, $v=b_-$. Suppose that $C$ is a coset of $H_\lambda$. Our next goal is to estimate $\dl (u,v)$.

\begin{figure}
 \centering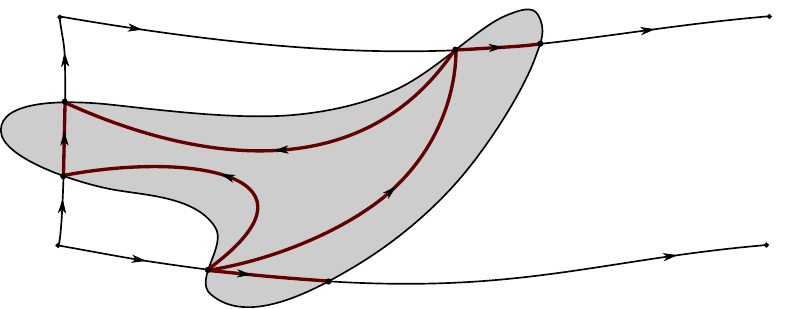\\
  \caption{}\label{fig05}
\end{figure}

As $u, v\in C$, there is an edge $e$ in $\G$ connecting $u$ to $v$ and labelled by an element of $H_\lambda$. We fix any geodesic $s$ connecting $1$ to $f$ in $\G$. It is easy to see that if $u\ne v$, then $e$ is an $H_\lambda$-subpath on the loop $p_1eq_1^{-1}s^{-1}$. We consider two cases. First assume that $e$ is an isolated component of $p_1eq_1^{-1}s^{-1}$. Then by Lemma \ref{C} we obtain $\dl (u,v)\le 4C$. Now suppose that $e$ is not an isolated component of $p_1eq_1^{-1}s^{-1}$. Note that $e$ can not be connected to a component of $p_1$. Indeed then $a$ would be connected to the same component of $p_1$, which contradicts the assumption that $p$ is a geodesic in $\G$. For the same reason $e$ cannot be connected to a component of $q_1$. Thus $e$ is connected to a component $d$ of $s$. Let $s=s_1ds_2$ and let $t_1$, $t_2$ be edges connecting $u$ to $d_-$ and $v$ to $d_+$, respectively, labelled by elements of $H_\lambda$ (Fig. \ref{fig05}). It is easy to see that  $t_1$ and $t_2$ are isolated in the loops $t_1s_1^{-1}p_1$ and $t_2s_2q_1$, respectively. Thus we have $\dl (u, d_-)\le 3C$ and $\dl (v, d_+)\le 3C$ by Lemma \ref{C}. Note also that since $d_{\pm}\in C\notin S(1,f;D)$, we have $\dl (d_-, d_+)\le D$. Therefore,
$$
\dl (u,v)\le \dl (u, d_-) +\dl (d_-, d_+) + \dl (d_+, v) \le 6C +D.
$$
Thus, in any case, we have
\begin{equation}\label{dluv}
\dl (u,v)\le 6C+D\le 3D
\end{equation}
(see (\ref{DC})).

Note that $$f= u\cdot  (u^{-1}v) \cdot (f^{-1}v)^{-1}.$$
Since $C\in S(f, fg;D)$, we have $f^{-1}C \in S(1,g; D)$. Thus both $u$ and $f^{-1}v$ are entrance points of $p$ in certain cosets from $S(1,g;D)$. Hence there are at most $|S(1,g;D)|\le 3|g|_{Y\cup \mathcal H} = 54\e + 33$ possibilities for $u$ and $f^{-1}v$. By (\ref{dluv}) the number of possibilities for $u^{-1}v$ does not exceed the number of elements in a ball of radius $3D$ in $H_\lambda $ with respect to the metric $\dl$. The latter number is finite by the definition of a hyperbolically embedded collection of subgroups. (We also use here our assumption that $|\Lambda|<\infty$. Thus the number of elements $f$ satisfying (\ref{e-close}) is bounded by a uniform constant.
\end{proof}

Lemmas \ref{he} and \ref{acylGy} prove Theorem \ref{XtoY}. To prove Theorem \ref{main} we only need one additional fact.

\begin{lem}\label{ah}
Let $G$ be a group, $H$ a subgroup of $G$, $X$ a subset of $G$. Suppose that $H$ is non-degenerate and $H\h (G,X)$. Then the action of $G$ on $\Gamma (G, X\sqcup H)$ is non-elementary.
\end{lem}

\begin{proof}
Since $H$ is non-degenerate, $G$ contains elements that act loxodromically on $\Gamma (G, X\sqcup H)$ by \cite[Corollary 6.12]{DGO}. Furthermore, by Theorem 6.14 from \cite{DGO} $G$ contains non-abelian free subgroups. In particular, $G$ is not virtually cyclic. Hence by Theorem \ref{class} the action of $G$ is non-elementary.
\end{proof}

\begin{proof}[Proof of Theorem \ref{main}]
Note that for a cobounded action of a group $G$ on a hyperbolic space $S$, we have $\Lambda (G)=\partial S$. Thus non-elementarity of the action is equivalent to $|\partial S|>2$. Now it is clear that (AH$_1$) implies  (AH$_2$). The implication (AH$_2$) $\Longrightarrow$  (AH$_3$) is obvious, and (AH$_3$) $\Longrightarrow$  (AH$_4$) was proved in \cite[Theorem 6.8]{DGO}.

It remains to show that (AH$_4$) $\Longrightarrow$  (AH$_1$). Let $H$ be a non-degenerate hyperbolically embedded subgroup of a group $G$. By Theorem \ref{XtoY} we can find a relative generating set $X$ of $G$ with respect to $H$ such that $H\h (G,X)$ and $\Gamma (G, X\sqcup H)$ is acylindrical. Since $H$ is non-degenerate, the action of $G$ on $\Gamma (G, X\sqcup H)$ is non-elementary by Lemma \ref{ah} and hence $|\partial \Gamma (G, X\sqcup H)|>2$.
\end{proof}


\section{Loxodromic elements}


The main goal of this section is to prove Theorem \ref{lox}. Recall that an isometry $g$ of a hyperbolic space $S$ is called \emph{elliptic} if some (equivalently, any) orbit of $g$ is bounded and \emph{loxodromic} if the map $\mathbb Z\to S$ defined by $n\mapsto g^ns$ is a quasi-isometry for some (equivalently, any) $s\in S$.

We will need the following three results. The first one in proved by Bowditch in \cite{Bow}. It also follows from Theorem \ref{class} applied to cyclic subgroups.

\begin{lem}[Bowditch]\label{bow}
Let $G$ be a group acting acylindrically on a hyperbolic space $S$. Then every element of $G$ is either elliptic or loxodromic.
\end{lem}

The next result is obtained in the course of proving Proposition 4.35 in \cite{DGO}.

\begin{lem}[Dahmani-Guirardel-Osin]\label{435}
Being hyperbolically embedded is a transitive relation. More precisely, if $H\h (G,X)$ and $K\h (H,Y)$, then $K\h (G, X\cup Y)$.
\end{lem}

The last auxiliary result is a particular case of Lemma 4.11 (b) \cite{DGO}. To state it, we extend the division operation to $\mathbb R_+ \cup \{\infty\} $ in the natural way: $\infty/\infty =1$, $\infty/r=\infty $ and $r/\infty =0$ for every $r\in (0, +\infty )$.

\begin{lem}\label{411}
Let $H\h (G,X)$ and let $\widehat\d $ be the associated relative metric on $H$. Then there exists a finite subset $Y\subseteq H$ such that the word metric $\d_Y$ is Lipschitz equivalent to $\widehat\d$. That is, the ratios $\d _Y(a,b)/\widehat\d(a,b)$ and $\widehat\d(a,b)/\d _Y(a,b)$ are uniformly bounded on $H\times H$ minus the diagonal. In particular, $\d_Y$ is finite if and only if $\widehat\d$ is.
\end{lem}

Now we are ready to prove the main result of this section.

\begin{proof}[Proof of Theorem \ref{lox}]
Obviously (L$_1$) implies (L$_2$) and (L$_2$) implies (L$_3$). It is also known that (L$_3$) implies (L$_4$), see \cite[Theorem 6.8]{DGO}. It remains to prove that (L$_4$) implies (L$_1$).

Let $E$ be a virtually cyclic subgroup containing $g$ such that $E\h G$. By Theorem \ref{XtoY} there exists a subset $X\subseteq G$ such that $E\h (G,X)$ and $\Gamma (G,X\sqcup E)$ is acylindrical. In particular $G$ is generated by $X\cup E$. Since $E$ is finitely generated,  we can assume that $G$ is generated by $X$ by Lemmas \ref{acyl} and \ref{X1X2} (adding generators of $E$ to $X$ if necessary).

Let $Y$ be a finite generating set of $E$. Since $E$ is hyperbolic we have $\{ 1\} \h (E,Y)$. By Lemma \ref{435} we obtain $\{ 1\} \h (G, X\cup Y)$. In particular, by definition this means that $\Gamma (G,X\cup Y)$ is hyperbolic. Again applying Lemma \ref{acyl} we obtain that $\Gamma (G,X)$ is hyperbolic.

Let us show that $\Gamma (G,X)$ is acylindrical. By Lemma \ref{411} there exists a finite subset $Z\subseteq E$ and a constant $K_0$ such that $\d_Z (a,b)\le K_0\widehat\d(a,b)$ for every $a,b\in E$, where $\widehat \d$ is the relative metric on $E$ associated to the embedding $E\h (G,X)$. Since $X$ generates $G$ and $Z$ is finite, there exists a constant $K$ such that
\begin{equation}\label{dd}
\d _X(a,b)\le K\widehat \d (a,b)
\end{equation}
for every $a,b\in E$. Fix $\e \in \mathbb N$. Since $\Gamma (G, X\sqcup E)$ is acylindrical there exists $R_0=R_0(\e)$ and $N_0=N_0(\e)$ such that for any $g\in G$ of length $|g|_{X\cup E}\ge R_0$, the set $\mathcal F_{\e }^{X\cup E}(g)$ of all $f\in G$ satisfying
$\max \{ |f|_{X\cup E}, |g^{-1}fg|_{X\cup E}\}  \le \e$ has at most $N_0$ elements. Let
\begin{equation}\label{RR0}
R=2CK(\e +1) R_0,
\end{equation}
where $C$ is the constant from Lemma \ref{C}. Without loss of generality we assume that $R$ is an integer.

By Lemma \ref{=R}, it suffices to show that there exists $N>0$ such that the set $\mathcal F_{\e}^X(g)$ of all elements $f$ satisfying
\begin{equation}\label{f}
\max \{ |f|_X, |g^{-1}fg|_X\}  \le \e
\end{equation}
has at most $N$ elements for every $g\in G$ of length
\begin{equation}\label{gX}
|g|_X= R.
\end{equation}
Note that for any $g\in G$ we have $\mathcal F_{\e}^X(g)\subseteq \mathcal F_{\e}^{X\cup E}(g)$ and hence $|\mathcal F_{\e}^X(g)|\le N_0$ whenever $|g|_{X\cup E}\ge R_0$. Thus it suffices to consider the case
\begin{equation}\label{gXE}
|g|_{X\cup E}< R_0.
\end{equation}

\begin{figure}
 \centering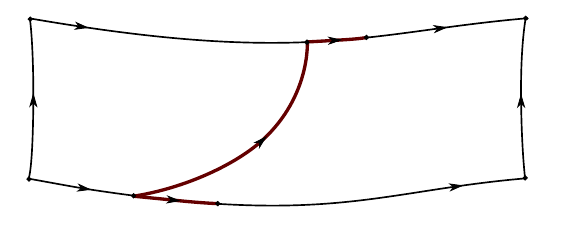\\
  \caption{}\label{fig08}
\end{figure}

Let $p$ be a geodesic in $\Gamma (G, X\sqcup E)$ with $p_-=1$, $p_+=g$. By (\ref{RR0}), (\ref{gX}) and (\ref{gXE}), there exists at least one $E$-component $d$ of $p$ such that $\d_X (d_-, d_+)> 2CK(\e +1)$. Further by (\ref{dd}) we obtain
\begin{equation}\label{de}
\widehat \d(d_-, d_+)> 2C(\e +1).
\end{equation}
Let $q=fp$. Thus $q$ is a geodesic in $\Gamma (G, X\sqcup E)$ such that $q_-=f$, $q_+=fg$, and $\Lab (q)\equiv \Lab(p)$. By (\ref{f}) there exist paths $s_1$ and $s_2$ connecting $1$ to $f$ and $g$ to $fg$, respectively (see Fig. \ref{fig08}), such that $\ell (s_i) \le \e$ and $\Lab (s_i)$ is a word in the alphabet $X$, $i=1,2$. Let $Q=s_1q s_2^{-1}p^{-1}$. We think of $Q$ as a geodesic $(2\e+2)$-gon whose sides are $p$, $q$, and the edges of $s_1$, $s_2$. Lemma \ref{C} and inequality (\ref{de}) imply that $d$ cannot be isolated in $Q$. Since $s_1$, $s_2$ do not contain $E$-components at all and $p$ is geodesic, $d$ must be connected to a component $e$ of $q$.

Let $p=p_1dp_2$, $q=q_1eq_2$, $u=d_-$, $v=e_-$, and let $c$ be the edge of $\Gamma (G, X\sqcup E)$ connecting $d_-$ to $e_-$ and labelled by an element of $E$. Since $p$ and $q$ are geodesic, $c$ is isolated in the loop $r=p_1cq_1^{-1}s_1^{-1}$. We can think of $r$ as a geodesic polygon with sides $p_1$, $c$, $q_1^{-1}$ and the edges of $s_1^{-1}$. Thus $r$ has at most $\e +3$ sides. Hence
\begin{equation}\label{dhuv}
\widehat d (u, v)=\widehat d (c_-, c_+) \le C(\e+3).
\end{equation}
Obviously we have $$f=u \cdot (u^{-1}v) \cdot (f^{-1}v)^{-1}.$$ Since $u$ and $f^{-1}v$ are elements represented by labels of initial subpaths of $p$ (recall that $\Lab (q)\equiv \Lab(p)$), there are at most $\ell(p)=|g|_{X\cup E}<R_0$ possibilities for these elements. Further by (\ref{dhuv}) the number of possible choices for $u^{-1}v$ is bounded by the number $B$ of elements in a ball of radius $C(\e+3)$ in $(E, \widehat d)$; this number is also finite as $(E, \widehat d)$ is locally finite. Therefore, $|\mathcal F_\e^X(g)|\le BR^2_0$ for every $g$ satisfying (\ref{gX}) and (\ref{gXE}). This completes the proof of acylindricity of  $\Gamma (G,X)$.

It remains to show that $g$ acts as a loxodromic element on $\Gamma (G,X)$. By Lemma \ref{bow} it suffices to show that $\langle g\rangle$ is unbounded with respect to $\d_{X}$. Suppose that there exists a constant $A$ such that $|g^n|_{X}\le A$ for every $n\in \mathbb Z$. Then for every $n\in \mathbb Z$ there exists a word $W_n$ in the alphabet $X$ of length at most $A$ representing $g^n$ in $G$. Obviously the path $p_n$ in $\Gamma (G, X\sqcup E)$ labelled by $W_n$ and connecting $1$ to $g^n$ is admissible (see Definition \ref{dl}). Hence $\widehat d(1,g^n)\le \ell(p)\le A$ for every $n$. Since $|g|=\infty$ this contradicts the local finiteness condition from the definition of a hyperbolically embedded subgroup.
\end{proof}

\begin{defn}\label{glox}
We say that an element $g\in G$ is \emph{generalized loxodromic} if it satisfies either of the equivalent conditions from Theorem \ref{lox}.
\end{defn}

The word ``generalized" is used here to distinguish these elements from loxodromic elements for a particular action. Note that, in general, generalized loxodromic elements may be elliptic for some actions. The simplest example is $G=F(x,y)$, the free group of rank $2$. Obviously every non-trivial element of $G$ is generalized loxodromic. However for the action of $G$ on the Bass-Serre tree associated to the free product decomposition $G=\langle x\rangle \ast \langle y\rangle $, all conjugates of powers of $x$ and $y$ are elliptic.

For $G=F_2$ as well as in many other cases, for a group $G$ there exists a {\it single} action on a hyperbolic space with respect to which \emph{all} generalized loxodromic elements are loxodromic. We call such actions \emph{universal}.
\begin{ex}
\begin{enumerate}
\item[(a)] One can show that an element of a mapping class group $MCG(\Sigma _g)$ of a closed surface of genus $g\ge 2$ is generalized loxodromic iff it is pseudo-Anosov. Indeed if $a\in MCG(\Sigma _g)$ is pseudo-Anosov, then it is generalized loxodromic since it acts loxodromically on the curve complex of $\Sigma _g$ and the action is acylindrical \cite{Bow}. On the other hand, it is not hard to prove that if $a$ is reducible, then there exists a positive integer $n$ such that $[C_G(a^n):\langle a^n\rangle]=\infty $; this is impossible if $G$ acts acylindrically on a hyperbolic space and the action of $a$ is loxodromic \cite{Bow}. Thus the action of $MCG(\Sigma _g)$ on the curve complex is universal.

\item[(b)] Any proper and cobounded action of a (hyperbolic) group on a hyperbolic space is universal.
Furthermore, if a group $G$ is relatively hyperbolic with respect to a collection of peripheral subgroups $\Hl$ and peripheral subgroups are not virtually cyclic and not acylindrically hyperbolic, then the action of $G$ on the corresponding  relative Cayley graph $\G$ is universal. This case includes fundamental groups of hyperbolic knot complements or limit groups.
\end{enumerate}
\end{ex}

The universal actions in the above examples are acylindrical. However, there are acylindrically hyperbolic groups that do not admit any universal acylindrical actions. To construct an example of such a group, we need several auxiliary results.

The first one is Lemma 3.3 from \cite[Chapter III.$\Gamma$]{BH}.

\begin{lem}[Bridson-Haefliger]
Let $S$ be a $\delta$-hyperbolic space and let $Y\subseteq S$ be a non-empty bounded subspace. Let $$r_Y=\inf \{ \rho >0 \mid Y\subseteq B(s, \rho) \; {\rm for \; some\;} s\in S\},$$ where $B(s,\rho)$ is the ball of radius $\rho$ centered at $s$. Then for all $\e>0$, the set $$C_\e =\{ s\in S \mid Y\subseteq B(s, r_y+\e)\}$$ has diameter less than $4\delta + 2\e$.
\end{lem}

This lemma immediately implies the following. (The idea is borrowed from the proof of Theorem 3.2 in Chapter III.$\Gamma$ of \cite{BH}.)

\begin{cor}\label{BHcor}
Assume that a group $T$ acts elliptically on a $\delta$-hyperbolic space $S$. Then there is $x\in S$ such that the diameter of the orbit $Tx$ is at most $4\delta +1$.
\end{cor}

\begin{proof}
Let $Y$ be some orbit of $T$. Since the action of $T$ is elliptic, $Y$ is bounded. Let $\e=1/2$ and let $C_\e$ be as in the lemma above. Note that the action of $T$ preserves $C_\e$ setwise. Hence for every $x\in C_\e$, we have $Tx\subseteq C_\e$ and thus the diameter of the orbit $Tx$ is at most $4\delta +1$ by the lemma.
\end{proof}

Recall that every loxodromic element $g$ of a group $G$ acting acylindrically on a hyperbolic space is contained in a unique maximal virtually cyclic subgroup $E(g)$ of $G$ \cite[Lemma 6.5]{DGO}.

\begin{lem}\label{egind}
Suppose that a group $G$ acts acylindrically on a hyperbolic space $S$. Then there exists $N\in \mathbb N$ such that for every loxodromic element $g\in G$, $E(g)$ contains a cyclic subgroup of index at most $N$.
\end{lem}

\begin{proof}
By \cite[Theorem 6.8]{DGO}, we have $E(g)\h G$. It is well-known (see, for example, \cite[Lemma 2.5]{FJ}) that every virtually cyclic group has a finite-by-cyclic subgroup of index at most $2$. Thus there exists a subgroup $E_0\le E(g)$ of index at most $2$ and a finite normal subgroup $T\le E_0$ such that $E_0/T$ is cyclic. Let $a\in E_0$ be an element whose image generates $E_0/T$. It is clear that $$|E(g):\langle a\rangle |\le 2|E_0:\langle a\rangle |= 2|T|.$$ Thus it suffices to bound $|T|$ from above.

Since $T$ is a finite normal subgroup of $E_0$, there is a positive integer $n$ such that $g^n$ commutes with $T$.
Let $\delta$ be the hyperbolicity constant of $S$. By Corollary \ref{BHcor}, there exists $x\in S$ such that the diameter of $Tx$ is at most $4\delta +1$. Let $R, N$ be the constants from the definition of acylindricity of the action corresponding to $\e =4\delta +1$. Since $g$ is a loxodromic element, there is $m\in \mathbb N$ such that the point $y=g^{mn}x$ satisfies $\d (x,y)\ge R$. Note that for every $t\in T$, we have $$\d (y, ty)=\d (g^{mn}x, tg^{mn}x)=\d (g^{mn}x, g^{mn}tx)=\d (x, tx)\le \e.$$ Hence $|T|\le N$.
\end{proof}

\begin{cor}\label{ci-lem}
Suppose that a group $G$ acts acylindrically on a hyperbolic space $S$. Then there exists $N\in \mathbb N$ such that for every loxodromic element $g\in G$, the centralizer $C_G(g)$ contains a cyclic subgroup of index at most $N$.
\end{cor}

\begin{proof}
For every $c\in C_G(g)$, the intersection $c^{-1}E(g)c \cap E(g)$ contains $\langle g\rangle $ and hence is infinite. By Lemma \ref{maln} we obtain $c\in E(g)$. Thus $C_G(g)\le E(g)$ and the claim follows from Lemma \ref{egind}
\end{proof}

\begin{ex}
Let $E_n=\langle x_n\rangle \times \mathbb Z/n\mathbb Z$ and let $G=E_1\ast E_2\ast \cdots $ be the free product of $E_n$ for all $n\in \mathbb N$. Then $G$ does not admit any universal action by either Lemma \ref{egind} or Corollary \ref{ci-lem}.
\end{ex}

The group $G$ in the above example is infinitely generated, which motivates the following.

\begin{prob}
Does every finitely generated (or finitely presented) group admit a universal acylindrical action on a hyperbolic space?
\end{prob}

It is possible that Dunwoody's example of an inaccessible group \cite{Dun}  provides the negative solution in the finitely generated case. For a discussion of Dunwoody's example in a similar context we refer to \cite{BDM}.

Even the following is unknown. If true, it would allow one to simplify many technical arguments from \cite{DGO} and some other papers.

\begin{prob}
Let $h_1, h_2$ be two generalized loxodromic elements of a group $G$. Does there always exist an acylindrical action of $G$ on a hyperbolic space such that both $h_1, h_2$ act loxodromically?
\end{prob}


\section{Some applications}


In this section we discuss some applications of Theorems \ref{class} -- \ref{lox}. In particular, we will prove Corollaries \ref{s-norm}--\ref{bgen}.

Corollary \ref{s-norm} is an immediate consequence of the following.

\begin{lem}\label{s-n}
Let $G$ be a group acting acylindrically and non-elementarily on a hyperbolic space $S$. Then every $s$-normal subgroup of $G$ acts non-elementarily.
\end{lem}

\begin{proof}
Arguing by contradiction, suppose that $H\le G$ is $s$-normal and elementary. Then by Theorem \ref{class} either $H$ is elliptic or $H$ is virtually cyclic and contains a generalized loxodromic element.

First assume that $H$ is elliptic, i.e., every orbit of $H$ is bounded. Fix any $x\in S$ and let $\e =\sup_{h\in H} \d(x, hx)$. Let $R=R(\e)$ be the constant from the acylindricity condition and $g\in G$ be any element such that $\d (x, gx) \ge R$ (such $g$ always exists as the action of $G$ is non-elementary). Then for every $h\in H\cap H^{g^{-1}}$, we have $\d (x, hx)\le \e $ and $$\d (gx, hgx) = \d (gx , gh^\prime x) = \d (x, h^\prime x) \le \e,$$ where $h^\prime =g^{-1}hg\in H$. Since $|H\cap H^{g^{-1}}|=|H\cap H^g|=\infty$, this contradicts acylindricity of the action.

Now assume that $H$ is virtually cyclic and contains a loxodromic element $h$. In particular, we have $|H:\langle h\rangle|<\infty$ and hence $\langle h\rangle$ is also $s$-normal in $G$. By condition (L$_4$) from Theorem \ref{lox} $h$ is contained in a virtually cyclic subgroup $E\h G$, which is almost malnormal in $G$ by Lemma \ref{maln}. Since $|E\cap E^g|\ge |\langle h\rangle \cap \langle h\rangle ^g|=\infty $ for every $g\in G$, we have $G=E$. This contradicts the assumption that $G$ acts non-elementarily.
\end{proof}

We mention one result of more algebraic flavor, which will be used below. It can also be derived from \cite{DGO}.

\begin{cor}\label{ar}
For every acylindrically hyperbolic group $G$, the following hold.
\begin{enumerate}
\item[(a)] $G$ has finite  amenable radical. In particular, the center of $G$ is finite.
\item[(b)] If $G$ decomposes as $G=G_1\times G_2$, then $|G_i|<\infty $ for $i=1$ or $i=2$.
\end{enumerate}
\end{cor}
\begin{proof}
Every acylindrically hyperbolic group contains non-abelian free subgroups by the standard ping-pong argument. In particular, acylindrically hyperbolic groups are non-amenable. Together with Corollary \ref{s-norm} this proves (a).

To prove (b) we note that if, say, $G_1$ is infinite, then it is acylindrically hyperbolic by Corollary \ref{s-norm}. In particular, it contains a generalized loxodromic element $g$. By Theorem \ref{lox}, $g$ is contained in a virtually cyclic hyperbolically embedded subgroup $E\h G$. By Lemma \ref{maln}, we have $G_2\le E$. Since $G/G_1\cong G_2$, and $[E:\langle g\rangle ]<\infty$, we obtain $|G_2|<\infty $.
\end{proof}

Corollary \ref{s-norm} is also useful in showing that certain groups are not acylindrically hyperbolic.

\begin{ex}
The Baumslag-Solitar groups $$BS(m,n)=\langle a,t\mid t^{-1}a^mt=a^n\rangle $$ are not acylindrically hyperbolic unless $m=n=0$ since $\langle a\rangle$ is $s$-normal in $BS(m,n)$.
\end{ex}

\begin{ex}
Let $R$ be an infinite integral domain, $K$ its field of fractions, $G$ a countable subgroup of $GL_n(R)$ containing a finite index subgroup of $EL_n(R)$, $n\ge 3$. Then $G$ contains a subgroup $H$ such that for every $g\in G$, the intersection $H^g\cap H$ contains an infinite amenable normal subgroup \cite{BFS}. In particular,
$H$ is $s$-normal in $G$ and is not acylindrically hyperbolic by Corollary \ref{ar} (a). Hence $G$ is not acylindrically hyperbolic by Corollary \ref{s-norm}. In particular, $SL_n(R)$ is not acylindrically hyperbolic for $n\ge 3$.
\end{ex}

\begin{proof}[Proof of Corollary \ref{ell}]
Let $\{ E_i\}_{i\in I}$ be a chain of elliptic subgroups of $G$ and let $E$ denote the union of the chain. Then $E$ is elliptic. Indeed otherwise it contains a loxodromic element $g$ by Theorem \ref{class}. Then $g\in E_i$ for some $i\in I$, which contradicts our assumption that $E_i$ is elliptic. Now the standard application of the Zorn lemma shows that every elliptic subgroup of $G$ is contained in a maximal elliptic subgroup.

Let $E$ denote an infinite maximal elliptic subgroup of $G$. Obviously $E$ is $s$-normal in its commensurator. Hence by Lemma \ref{s-n} $Comm_G(E)$ is elementary. Since $E$ is maximal, we obtain $E=Comm_G(E)$.
\end{proof}

In the proof of the next result we will use some notions from Section 4.

\begin{proof}[Proof of Proposition \ref{bgen}]
Let $G$ be an acylindrically hyperbolic group. By Theorem \ref{main}, there is a non-degenerate hyperbolically embedded subgroup $H\h G$. Further by Theorem \ref{XtoY}, there exists a subset $X\subseteq G$ such that $H\h (G,X)$ and the corresponding Cayley graph $\Gamma(G, X\sqcup H)$ is acylindrical. Let $G=G_1\ldots G_n$. Suppose that none of $G_1, \ldots , G_n$ is acylindrically hyperbolic. Then by Theorem \ref{class}, every $G_i$ is either elliptic with respect to the action of $G$ on the Cayley graph $\Gamma(G, X\sqcup H)$ or contains a cyclic subgroup of finite index generated by a loxodromic element.

Suppose that $G_i$ is of the latter type for some $i\in \{ 1, \ldots, n\}$. That is, there exists a loxodromic (with respect to the action on $\Gamma(G, X\sqcup H) $) element $x\in G_i$ such that
\begin{equation}\label{finind}
|G_i:\langle x\rangle|<\infty .
\end{equation}

By Lemmas \ref{acyl} and \ref{X1X2}, we can assume that $x\in X$. Since the action of $x$ on $\Gamma(G, X\sqcup H)$ is loxodromic, there exist $\mu\ge 1$ and $b\ge 0$ such that for every $n\in \mathbb N$, every path in $\Gamma(G, X\sqcup H)$ labelled by $x^n$ is $(\mu, b)$-quasi-geodesic. Let $D\ge 3C(\mu,b)$, where $C(\mu,b)$ is the constant from Lemma \ref{C}. Then for
every $n\in \mathbb N$, we have
\begin{equation}\label{S1x}
S(1,x^n;D)=\emptyset
\end{equation}
(see Section 4 for definitions). Indeed if $C\in  S(1,x^n;D)$, then the path $p$ in $\Gamma(G, X\sqcup H)$ starting at $1$ and labelled by the word $x^n$ in the alphabet $X$ should penetrate $C$ by Lemma \ref{sep}. However this is impossible since $p$ does not have any edges labelled by elements of $\mathcal H$. Thus (\ref{S1x}) holds and therefore we have $x^n\in Y$ for all $n\in \mathbb N$, where $Y$ is defined by (\ref{Y}).

We now pass to $\Gamma(G, Y\sqcup H)$. By Lemmas \ref{he} and \ref{acylGy}, $H\h (G,Y)$ and $\Gamma(G, Y\sqcup H)$ is acylindrical. Since $x^n\in Y$ for every $n\in \mathbb N$, $\langle x\rangle$ is elliptic with respect to the action on $\Gamma(G, Y\sqcup H)$. From this and (\ref{finind}) one can easily derive that $G_i$ acts elliptically on $\Gamma(G, Y\sqcup H)$. Note also that since $X\subseteq Y$ (see Lemma \ref{hypGy}), all subgroups that act elliptically on $\Gamma(G, X\sqcup H)$ also act elliptically on $\Gamma(G, Y\sqcup H)$. Iterating this process we can find a subset $Z\subseteq G$ such that $H\h (G,Z)$ and all subgroups $G_1, \ldots , G_n$ have bounded orbits in $\Gamma(G, Z\sqcup H)$. Hence so does $G=G_1\ldots G_n$. However this contradicts Lemma \ref{ah}.
\end{proof}


\section{Appendix: A brief survey}


Many general results about acylindrically hyperbolic groups were proved in \cite{BBF,BF,DGO,Ham,Hull,HO,Sis} and other papers under different assumptions, which are now known to be equivalent to acylindrical hyperbolicity. The purpose of this section is to bring these results together; if necessary, we reformulate them using the language of our paper.

We begin with examples of acylindrically hyperbolic groups. Obviously every proper and cobounded action is acylindrical. In particular, this applies to the action of any finitely generated group on its Cayley graph with respect to a finite generating set. Thus every non-elementary hyperbolic group is acylindrically hyperbolic. More generally, non-virtually-cyclic relatively hyperbolic groups with proper peripheral subgroups are acylindrically hyperbolic. In the latter case the action on the relative Cayley graph is acylindrical, see Proposition \ref{rhg}. Below we discuss some less obvious examples.

\begin{enumerate}
\item[(a)] The mapping class group $MCG(\Sigma_{g,p})$ of a closed surface of genus $g$ with $p$ punctures is acylindrically hyperbolic unless $g=0$ and $p\le 3$ (in these exceptional cases, $MCG(\Sigma_{g,p})$ is finite). For $(g,p)\in \{ (0,4), (1,0), (1,1)\} $ this follows from the fact that $MCG(\Sigma_{g,p})$ is non-elementary hyperbolic in these cases. For all other values of $(g,p)$ this follows from hyperbolicity of the curve complex $\mathcal C(\Sigma_{g,p})$ of $\Sigma_{g,p}$ first proved by Mazur and Minsky \cite{MM} and acylindricity of the action of $MCG(\Sigma_{g,p})$ on $\mathcal C(\Sigma_{g,p})$, which is due to Bowditch \cite{Bow}.
\item[(b)] Let $n\ge 2$ and let $F_n$ be the free group of rank $n$. Bestvina and Feighn \cite{BFe} proved that for every fully irreducible automorphism $f\in Out(F_n)$ there exists a hyperbolic graph such that $Out(F_n)$ acts on it and the action of $f$ satisfies the weak proper discontinuity condition. Thus $Out(F_n)$ satisfies condition (AH$_3$) and hence is acylindrically hyperbolic.
\item[(c)] Hamenst\"adt proved that every group acting elementary and properly on a proper hyperbolic space of bounded growth is virtually nilpotent (see the proof of Proposition 7.1 in \cite{Ham}). Note that, in general, proper actions may not be acylindrical (see Example \ref{ex1}), but they do satisfy the weak proper discontinuity condition for every loxodromic element. It follows that every group acting properly on a proper hyperbolic space of bounded growth is either virtually nilpotent or acylindrically hyperbolic.
\item[(d)] Sisto \cite{Sis} showed that if a group $G$ acts properly on a proper $CAT(0)$ space, then every rank $1$ element of $G$ is contained in a hyperbolically embedded virtually cyclic subgroup. In particular, such a group $G$ is either virtually cyclic or acylindrically hyperbolic. Together with the work of Caprace--Sageev \cite{CS} and Corollary \ref{ar}, this implies the following alternative for right angled Artin groups: every RAAG is either cyclic, or directly decomposable, or acylindrically hyperbolic. A similar result holds for graph products of groups (or, even more generally, subgroups of graph products) \cite{MO}.
\item[(e)] In \cite{MO}, Minasyan and the author show that many fundamental groups of graphs of groups satisfy (AH$_3$) and hence are acylindrically hyperbolic. As an application, it is shown that every $1$-relator group with at least $3$ generators is acylindrically hyperbolic. Yet another corollary is that for every field $k$, the automorphism group $Aut\, k[x,y]$ of the polynomial algebra $k[x,y]$ is acylindrically hyperbolic.
\item[(f)] For every compact $3$-manifold $M$, the fundamental group $\pi_1(M)$ is either virtually polycyclic, or acylindrically hyperbolic, or contains a normal subgroup $N\cong \mathbb Z$ such that $\pi_1(M)/N$ is acylindrically hyperbolic \cite{MO}.
\end{enumerate}

Now we turn to a brief discussion of some known properties of acylindrically hyperbolic groups. Many of them were first established for particular subclasses of acylindrically hyperbolic groups such as hyperbolic groups, non-trivial free products, relatively hyperbolic groups, mapping class groups, $Out(F_n)$, etc. For a survey of these (quite numerous) particular results, further details, and motivation we refer to \cite{DGO} and other papers cited in this section.

We begin with a general comment. According to Theorem \ref{main}, every acylindrically hyperbolic group contains non-degenerate hyperbolically embedded subgroups. Furthermore, by \cite[Theorem 2.23]{DGO} these subgroups can be chosen virtually free. This allows one to extend many results about relatively hyperbolic groups to the case of acylindrically hyperbolic ones. The paper \cite{DGO} by Dahmani, Guirardel, and the author provide convenient tools for such an extension. One example is the group theoretic Dehn filling theorem, which was originally  proved in \cite{Osi07} in the context of relatively hyperbolic groups (and independently in \cite{GM} in the torsion free case), and then generalized to groups with hyperbolically embedded subgroups in \cite{DGO}. For details and relation to $3$-dimensional topology and Thurston's work we refer to \cite{Osi07}.

One application of Dehn filling in hyperbolically embedded subgroups is the proof of following theorem, which can be thought of as an indication of algebraic and model theoretic ``largeness" of acylindrically hyperbolic groups. Recall that a group $G$ is \emph{$SQ$-universal} if every countable group embeds in a quotient of $G$. For a discussion of stability and superstability of elementary theories we refer to \cite{Shel} and the survey \cite{W}.

\begin{thm}[Dahmani--Guirardel--Osin {\cite{DGO}}]\label{T1}
Suppose that a group $G$ is acylindrically hyperbolic. Then the following hold.
\begin{enumerate}
\item[(a)] $G$ is $SQ$-universal, i.e., every countable group embeds in a quotient of $G$.
\item[(b)] The elementary theory of $G$ is not superstable.
\end{enumerate}
\end{thm}

If elliptic subgroups of $G$ are reducible in a certain sense, Theorem \ref{T1} often yields a classification of ``small" subgroups of $G$. For instance, one can show
that every subgroup of a mapping class group of a punctured closed surface is either virtually abelian or $SQ$-universal
\cite{DGO}. This generalizes the Tits alternative for subgroups of mapping class groups, as well as various results about non-embedability of higher rank lattices. For the latter application, one need to combine Theorem \ref{T1} with two facts: the Margulis Theorem about normal subgroups of higher rank lattices in semi-simple Lie groups and the easy observation that every countable $SQ$-universal group has uncountably many normal subgroups. For more details we refer to the discussion following Corollary 2.31 in \cite{DGO}.

The next result is of more analytic flavor. For a group $G$, by a \emph{normed $G$-module} we mean a normed vector space $V$ endowed with a (left) action of the group $G$ by isometries. Recall that a map $q\colon G\to V$ is called a \emph{1-quasi-cocycle} if there exists a constant $\e>0$ such that for every $f,g\in G$ we have $$\| q(fg)-q(f)-fq(g)\|\le \e .$$   The vector space of all $1$-quasi-cocycles on $G$ with values in $V$ is denoted by $\QZ (G,V)$. The study of $1$-quasi-cocycles is partially motivated by applications to bounded cohomology, of group von Neumann algebras, measure equivalence and orbit equivalence of groups, and low dimensional topology (see \cite{Cal,ChSi,Ham,Mon,Pop} and references therein).

Given a subgroup $H\le G$, by an $H$-submodule of a $G$-module $V$ we mean any $H$-invariant subspace of $V$ with the induced action of $H$. If $H$ is a subgroup of $G$, the restriction functor ${\rm res}_H$ obviously maps $\QZ(G,V)$ to $\QZ(H,V)$. In general, it is not invertible. However, the following result proved in \cite{HO} shows that one can extend quasi-cocycles from a hyperbolically embedded subgroup to the whole group, possibly after a bounded perturbation. By $\| \cdot \|_\infty$ we denote the $sup$-norm on $\QZ(G,V)$; that is, $\| q\|_{\infty}=\sup_{g\in G} \| q(g)\| $ for $q\in \QZ(G,V)$.

\begin{thm}\label{ind1}
Let $V$ be a $G$ module, $U$ an $H$-submodule of $V$. There exists a linear map $$\kappa\colon \QZ (H, U)\to \QZ (G,V)$$ such that for any $q\in \QZ(H,U)$, we have $\| {\rm res}_H(\kappa (q)) -q\|_\infty <\infty .$
\end{thm}

Applying Theorem \ref{ind1} to quasimorphisms and using Bavard duality, Hull and the author also showed that hyperbolically embedded subgroups are undistorted with respect to the stable commutator length \cite{HO}.

The next result can also be derived from Theorem \ref{ind1}, see \cite{HO}. It opens the door for Monod-Shalom rigidity theory for group actions on spaces with measure \cite{MS}. In the case $V=\mathbb R$ it was first proved by Bestvina and Fujiwara in \cite{BF} for groups acting weakly properly discontinuously on hyperbolic spaces. The first proof for $\ell^p$-spaces was given in \cite{Ham} under the assumption of weak acylindricity. Finally, Bestvina, Bromberg, and Fujiwara extended this theorem to uniformly convex Banach $G$-modules in \cite{BBF}.

\begin{thm}[Bestvina--Fujiwara \cite{BF}, Hamenst\"adt \cite{Ham}]
Suppose that a group $G$ is acylindrically hyperbolic. Let $V=\mathbb R$ or $V=\ell^p(G)$ for some $p\in [1, +\infty)$. Then the kernel of the natural map $H^2(G, V) \to H^2_b(G, V)$ is infinite dimensional. In particular, ${\rm dim\,} H^2_b(G, V)=\infty $.
\end{thm}

The next theorem relates algebraic properties of acylindrically hyperbolic groups to some basic properties of their $C^\ast$-algebras and von Neumann algebras. Recall that a group $G$ is \emph{inner amenable} if there exists a finitely additive conjugacy invariant probability measure on $G\setminus \{ 1\}$. Inner amenability is closely related to the Murray--von Neumann property $\Gamma $ for operator algebras. In particular, if $G$ is not inner amenable, the von Neumann algebra $W^\ast (G)$ of $G$ does not have property $\Gamma $ \cite{Efr}. For further details and motivation we refer to \cite{BeHa}. We recall that every acylindrically hyperbolic group $G$ contains a maximal normal finite subgroup, denoted $K(G)$ \cite[Theorem 2.23]{DGO}.

\begin{thm}[Dahmani--Guirardel--Osin, {\cite{DGO}}]\label{T2-DGO}
For any countable acylindrically hyperbolic group $G$, the following conditions are equivalent.
\begin{enumerate}
\item[(a)] $K(G)=\{ 1\}$.
\item[(b)] $G$ has infinite conjugacy classes (equivalently, the von Neumann algebra $W^\ast (G)$ of $G$ is a $II_1$ factor).
\item[(c)] $G$ is not inner amenable. In particular, $W^\ast (G)$ does not have property $\Gamma $ of Murray and von Neumann.
\item[(d)] The reduced $C^\ast $-algebra of $G$ is simple with unique trace.
\end{enumerate}
\end{thm}

The next theorem was proved by Sisto in \cite{Sis} in the context of groups containing weakly contracting elements. It is shown in \cite{Sis} that every loxodromic element of a group $G$ acting acylindrically on a hyperbolic space is weakly contracting and every weakly contracting element is contained in a virtually cyclic hyperbolically embedded subgroup of $G$.  Thus, according to Theorem \ref{lox}, an element of a group $G$ is weakly contracting if and only if it is generalized loxodromic in the sense of our Definition \ref{glox}. In particular, the class of groups which are not virtually cyclic and contain weakly contracting elements coincides with the class of acylindrically hyperbolic group.

Let $G$ be a group generated by a finite symmetric set $X$. The \emph{simple random walk} on $G$ is a Markov chain with the set of states $G$, initial state $1$, and transition probability from $g$ to $h$ equal to $1/|X|$ if $g^{-1}h\in X$ and $0$ otherwise.

\begin{thm}[Sisto, {\cite{Sis}}]
For any finitely generated acylindrically hyperbolic group, the probability that the simple random walk arrives at a generalized loxodromic element in $n$ steps is at least $1- O(\e ^n)$ for some $\e \in (0,1)$.
\end{thm}

Informally, this theorem says that generic elements of $G$ are generalized loxodromic.

Another direction is explored in \cite{Hull}, where the small cancellation theory in hyperbolic and relatively hyperbolic groups developed in \cite{Ols93,Osi10} was generalized in the context of acylindrically hyperbolic groups. We do not go into details here and only give one exemplary result. By $\pi(G)$ we denote the set of orders of elements of a group $G$.

\begin{thm}[Hull, \cite{Hull}]
Let $G$ be a countable acylindrically hyperbolic group. Then $G$ has an infinite, finitely generated quotient $C$ such that any two elements of $C$ are conjugate if and only if they have the same order. Moreover, and $\pi(C)=\pi(G)$. In particular, if $G$ is torsion free, then $C$ has two conjugacy classes.
\end{thm}

Hull also showed in  \cite{Hull} that the quotient group with $2$ conjugacy classes in the above theorem can be constructed so that it admits a non-discrete topology. For motivation we refer to the discussion in the last section of \cite{KOO}. Other applications of the small cancellation theory include results about Kazhdan constants, Frattini subgroups of acylindrically hyperbolic groups, etc; see \cite{Hull} for details.

Yet another application of Dehn filling and small-cancellation-like techniques is given in \cite{AMS}, where Antolin, Minasyan, and Sisto obtained a classification of commensurating endomorphisms of acylindrically hyperbolic groups. Their results are similar to the classification obtained in \cite{MO10} for relatively hyperbolic groups. Recall that an endomorphism $\phi$ of a group $G$ is said to be \emph{commensurating}, if for every $g\in G$, some non-zero power of $\phi(g)$ is conjugate to a non-zero power of $g$. Given an acylindrically hyperbolic group $G$, they show that any commensurating endomorphism of $G$ is inner modulo a small perturbation. We mention just one exciting application of this fact and refer to \cite{AMS} for other applications and more details.

\begin{thm}[Antolin--Minasyan--Sisto, \cite{AMS}]
Let $G$ be the fundamental group of a compact $3$-manifold. Then $Out(G)$ is residually finite.
\end{thm}

Finally we want to mention a result of Sisto \cite{Sis13}, stating that generalized loxodromic elements of finitely generated acylindrically hyperbolic groups are Morse. We do not discuss the definition of a Morse element here and only mention one corollary, which is a combination of Sisto's result and a theorem of Drutu and Sapir, stating that existence of Morse elements implies existence of cut points in all asymptotic cones.

\begin{thm}[Sisto \cite{Sis13}]
Every finitely generated acylindrically hyperbolic group has cut points in all asymptotic cones.
\end{thm}

Potentially this result can be useful in the study of various questions about geometric rigidity of acylindrically hyperbolic groups. For definitions and results about groups with cut points in asymptotic cones see \cite{DS,OOS} and references therein.

\vspace{1cm}

\noindent \textbf{Denis Osin: } Department of Mathematics, Vanderbilt University, Nashville 37240, U.S.A.\\
E-mail: \emph{denis.v.osin@vanderbilt.edu}

\end{document}

%% file: fig01.pdf_tex

\begingroup
  \makeatletter
  \providecommand\color[2][]{%
    \errmessage{(Inkscape) Color is used for the text in Inkscape, but the package 'color.sty' is not loaded}
    \renewcommand\color[2][]{}%
  }
  \providecommand\transparent[1]{%
    \errmessage{(Inkscape) Transparency is used (non-zero) for the text in Inkscape, but the package 'transparent.sty' is not loaded}
    \renewcommand\transparent[1]{}%
  }
  \providecommand\rotatebox[2]{#2}
  \ifx\svgwidth\undefined
    \setlength{\unitlength}{315.6032959pt}
  \else
    \setlength{\unitlength}{\svgwidth}
  \fi
  \global\let\svgwidth\undefined
  \makeatother
  \begin{picture}(1,0.33823627)%
    \put(0,0){\includegraphics[width=\unitlength]{fig01.pdf}}%
    \put(0.14529654,0.04519912){\color[rgb]{0,0,0}\makebox(0,0)[lb]{\smash{$x$}}}%
    \put(-0.00085154,0.26912947){\color[rgb]{0,0,0}\makebox(0,0)[lb]{\smash{$gx$}}}%
    \put(0.8081446,0.27761601){\color[rgb]{0,0,0}\makebox(0,0)[lb]{\smash{$gy$}}}%
    \put(0.94235742,0.06133826){\color[rgb]{0,0,0}\makebox(0,0)[lb]{\smash{$y$}}}%
    \put(0.63988969,0.05192185){\color[rgb]{0,0,0}\makebox(0,0)[lb]{\smash{$z$}}}%
    \put(0.64470778,0.25526651){\color[rgb]{0,0,0}\makebox(0,0)[lb]{\smash{$gz$}}}%
    \put(0.45566446,0.26567467){\color[rgb]{0,0,0}\makebox(0,0)[lb]{\smash{$z_0$}}}%
    \put(0.11480396,0.15423148){\color[rgb]{0,0,0}\makebox(0,0)[lb]{\smash{$\le \varepsilon$}}}%
    \put(0.86346586,0.16309549){\color[rgb]{0,0,0}\makebox(0,0)[lb]{\smash{$\le \varepsilon$}}}%
    \put(0.58066795,0.16654425){\color[rgb]{0,0,0}\makebox(0,0)[lb]{\smash{$\le 2\delta + \varepsilon$}}}%
    \put(0.53468466,0.32403727){\color[rgb]{0,0,0}\makebox(0,0)[lb]{\smash{$\le 2\delta + 2\varepsilon$}}}%
    \put(0.38955471,0.00133673){\color[rgb]{0,0,0}\makebox(0,0)[lb]{\smash{$R$}}}%
    \put(0.39185383,0.09675224){\color[rgb]{0,0,0}\makebox(0,0)[lb]{\smash{$p$}}}%
    \put(0.26310039,0.31057494){\color[rgb]{0,0,0}\makebox(0,0)[lb]{\smash{$q$}}}%
  \end{picture}%
\endgroup

%% file: fig0.pdf_tex

\begingroup
  \makeatletter
  \providecommand\color[2][]{%
    \errmessage{(Inkscape) Color is used for the text in Inkscape, but the package 'color.sty' is not loaded}
    \renewcommand\color[2][]{}%
  }
  \providecommand\transparent[1]{%
    \errmessage{(Inkscape) Transparency is used (non-zero) for the text in Inkscape, but the package 'transparent.sty' is not loaded}
    \renewcommand\transparent[1]{}%
  }
  \providecommand\rotatebox[2]{#2}
  \ifx\svgwidth\undefined
    \setlength{\unitlength}{433.25592422pt}
  \else
    \setlength{\unitlength}{\svgwidth}
  \fi
  \global\let\svgwidth\undefined
  \makeatother
  \begin{picture}(1,0.52751731)%
    \put(0,0){\includegraphics[width=\unitlength]{fig0.pdf}}%
    \put(0.30789297,0.25624641){\color[rgb]{0,0,0}\makebox(0,0)[lb]{\smash{$\Gamma_H$}}}%
    \put(0.31115711,0.07051609){\color[rgb]{0,0,0}\makebox(0,0)[lb]{\smash{$x^{-1}\Gamma_H$}}}%
    \put(0.31050429,0.44165028){\color[rgb]{0,0,0}\makebox(0,0)[lb]{\smash{$x\Gamma _H$}}}%
    \put(0.15121363,0.33034268){\color[rgb]{0,0,0}\makebox(0,0)[lb]{\smash{$x$}}}%
    \put(0.20277635,0.33082313){\color[rgb]{0,0,0}\makebox(0,0)[lb]{\smash{$x$}}}%
    \put(0.25369717,0.32919105){\color[rgb]{0,0,0}\makebox(0,0)[lb]{\smash{$x$}}}%
    \put(0.25271788,0.19046458){\color[rgb]{0,0,0}\makebox(0,0)[lb]{\smash{$x$}}}%
    \put(0.20277635,0.19079099){\color[rgb]{0,0,0}\makebox(0,0)[lb]{\smash{$x$}}}%
    \put(0.15218198,0.19209666){\color[rgb]{0,0,0}\makebox(0,0)[lb]{\smash{$x$}}}%
    \put(0.14566457,0.2601634){\color[rgb]{0,0,0}\makebox(0,0)[lb]{\smash{$h_1$}}}%
    \put(0.21584385,0.25787848){\color[rgb]{0,0,0}\makebox(0,0)[lb]{\smash{$h_2$}}}%
    \put(0.0879275,0.36624836){\color[rgb]{0,0,0}\makebox(0,0)[lb]{\smash{. . .}}}%
    \put(0.25436086,0.36526914){\color[rgb]{0,0,0}\makebox(0,0)[lb]{\smash{. . .}}}%
    \put(0.25403442,0.1371048){\color[rgb]{0,0,0}\makebox(0,0)[lb]{\smash{. . .}}}%
    \put(0.08890675,0.14004256){\color[rgb]{0,0,0}\makebox(0,0)[lb]{\smash{. . .}}}%
    \put(0.25318824,0.49238002){\color[rgb]{0,0,0}\makebox(0,0)[lb]{\smash{. . .}}}%
    \put(0.08871336,0.49368567){\color[rgb]{0,0,0}\makebox(0,0)[lb]{\smash{. . .}}}%
    \put(0.08871336,0.04127405){\color[rgb]{0,0,0}\makebox(0,0)[lb]{\smash{. . .}}}%
    \put(0.25449387,0.0393156){\color[rgb]{0,0,0}\makebox(0,0)[lb]{\smash{. . .}}}%
    \put(0.59194638,0.34752468){\color[rgb]{0,0,0}\makebox(0,0)[lb]{\smash{. . . }}}%
    \put(0.88825663,0.34752468){\color[rgb]{0,0,0}\makebox(0,0)[lb]{\smash{. . .}}}%
    \put(0.51507043,0.44302417){\color[rgb]{0,0,0}\makebox(0,0)[lb]{\smash{. . .}}}%
    \put(0.96091471,0.43943692){\color[rgb]{0,0,0}\makebox(0,0)[lb]{\smash{. . .}}}%
    \put(0.87922144,0.49990775){\color[rgb]{0,0,0}\makebox(0,0)[lb]{\smash{. . .}}}%
    \put(0.73870036,0.50042021){\color[rgb]{0,0,0}\makebox(0,0)[lb]{\smash{. . .}}}%
    \put(0.60699695,0.49119584){\color[rgb]{0,0,0}\makebox(0,0)[lb]{\smash{. . .}}}%
    \put(0.88813875,0.18192882){\color[rgb]{0,0,0}\makebox(0,0)[lb]{\smash{. . .}}}%
    \put(0.96079682,0.09001659){\color[rgb]{0,0,0}\makebox(0,0)[lb]{\smash{. . .}}}%
    \put(0.87807234,0.03428172){\color[rgb]{0,0,0}\makebox(0,0)[lb]{\smash{. . .}}}%
    \put(0.73966687,0.02969577){\color[rgb]{0,0,0}\makebox(0,0)[lb]{\smash{. . .}}}%
    \put(0.60687901,0.03825768){\color[rgb]{0,0,0}\makebox(0,0)[lb]{\smash{. . .}}}%
    \put(0.59222655,0.18079345){\color[rgb]{0,0,0}\makebox(0,0)[lb]{\smash{. . . }}}%
    \put(0.51478812,0.08293444){\color[rgb]{0,0,0}\makebox(0,0)[lb]{\smash{. . .}}}%
    \put(0.8746591,0.26002312){\color[rgb]{0,0,0}\makebox(0,0)[lb]{\smash{$\Gamma_H$}}}%
    \put(0.7766041,0.40793664){\color[rgb]{0,0,0}\makebox(0,0)[lb]{\smash{$x\Gamma_H$}}}%
    \put(0.77007283,0.11322256){\color[rgb]{0,0,0}\makebox(0,0)[lb]{\smash{$x^{-1}\Gamma_H$}}}%
    \put(0.68283493,0.34458917){\color[rgb]{0,0,0}\makebox(0,0)[lb]{\smash{$x$}}}%
    \put(0.86305141,0.18713953){\color[rgb]{0,0,0}\makebox(0,0)[lb]{\smash{$x$}}}%
    \put(0.77138776,0.18185126){\color[rgb]{0,0,0}\makebox(0,0)[lb]{\smash{$x$}}}%
    \put(0.68060537,0.18185126){\color[rgb]{0,0,0}\makebox(0,0)[lb]{\smash{$x$}}}%
    \put(0.86187617,0.33844342){\color[rgb]{0,0,0}\makebox(0,0)[lb]{\smash{$x$}}}%
    \put(0.77226908,0.34520068){\color[rgb]{0,0,0}\makebox(0,0)[lb]{\smash{$x$}}}%
    \put(0.76252196,0.25499366){\color[rgb]{0,0,0}\makebox(0,0)[lb]{\smash{$1$}}}%
    \put(0.10997274,0.44342176){\color[rgb]{0,0,0}\makebox(0,0)[lb]{\smash{$xh_1$}}}%
    \put(0.24175036,0.44342176){\color[rgb]{0,0,0}\makebox(0,0)[lb]{\smash{$xh_2$}}}%
  \end{picture}%
\endgroup

%% file: fig03.pdf_tex

\begingroup
  \makeatletter
  \providecommand\color[2][]{%
    \errmessage{(Inkscape) Color is used for the text in Inkscape, but the package 'color.sty' is not loaded}
    \renewcommand\color[2][]{}%
  }
  \providecommand\transparent[1]{%
    \errmessage{(Inkscape) Transparency is used (non-zero) for the text in Inkscape, but the package 'transparent.sty' is not loaded}
    \renewcommand\transparent[1]{}%
  }
  \providecommand\rotatebox[2]{#2}
  \ifx\svgwidth\undefined
    \setlength{\unitlength}{284.9262207pt}
  \else
    \setlength{\unitlength}{\svgwidth}
  \fi
  \global\let\svgwidth\undefined
  \makeatother
  \begin{picture}(1,0.67428256)%
    \put(0,0){\includegraphics[width=\unitlength]{fig03.pdf}}%
    \put(-0.00094323,0.35892413){\color[rgb]{0,0,0}\makebox(0,0)[lb]{\smash{$x$}}}%
    \put(0.28486593,0.37677215){\color[rgb]{0,0,0}\makebox(0,0)[lb]{\smash{$y$}}}%
    \put(0.30061039,0.22154837){\color[rgb]{0,0,0}\makebox(0,0)[lb]{\smash{$t$}}}%
    \put(0.30601895,0.48223783){\color[rgb]{0,0,0}\makebox(0,0)[lb]{\smash{$z$}}}%
    \put(0.8144175,0.6585548){\color[rgb]{0,0,0}\makebox(0,0)[lb]{\smash{$h^{-1}x$}}}%
    \put(0.95660061,0.35459735){\color[rgb]{0,0,0}\makebox(0,0)[lb]{\smash{$gx$}}}%
    \put(0.81742207,0.02840481){\color[rgb]{0,0,0}\makebox(0,0)[lb]{\smash{$hx$}}}%
    \put(0.46430699,0.00484775){\color[rgb]{0,0,0}\makebox(0,0)[lb]{\smash{$hy$}}}%
    \put(0.54122821,0.1774392){\color[rgb]{0,0,0}\makebox(0,0)[lb]{\smash{$hz$}}}%
    \put(0.33402287,0.41192739){\color[rgb]{0,0,0}\makebox(0,0)[lb]{\smash{$\le 4\delta$}}}%
    \put(0.3328465,0.29038386){\color[rgb]{0,0,0}\makebox(0,0)[lb]{\smash{$\le 4\delta$}}}%
    \put(0.53824918,0.07903186){\color[rgb]{0,0,0}\makebox(0,0)[lb]{\smash{$\le 4\delta$}}}%
  \end{picture}%
\endgroup

%% file: fig02.pdf_tex

\begingroup
  \makeatletter
  \providecommand\color[2][]{%
    \errmessage{(Inkscape) Color is used for the text in Inkscape, but the package 'color.sty' is not loaded}
    \renewcommand\color[2][]{}%
  }
  \providecommand\transparent[1]{%
    \errmessage{(Inkscape) Transparency is used (non-zero) for the text in Inkscape, but the package 'transparent.sty' is not loaded}
    \renewcommand\transparent[1]{}%
  }
  \providecommand\rotatebox[2]{#2}
  \ifx\svgwidth\undefined
    \setlength{\unitlength}{310.79320068pt}
  \else
    \setlength{\unitlength}{\svgwidth}
  \fi
  \global\let\svgwidth\undefined
  \makeatother
  \begin{picture}(1,0.55487783)%
    \put(0,0){\includegraphics[width=\unitlength]{fig02.pdf}}%
    \put(0.01247623,0.14604772){\color[rgb]{0,0,0}\makebox(0,0)[lb]{\smash{$x$}}}%
    \put(-0.00086472,0.34574894){\color[rgb]{0,0,0}\makebox(0,0)[lb]{\smash{$gx$}}}%
    \put(0.82705086,0.54045908){\color[rgb]{0,0,0}\makebox(0,0)[lb]{\smash{$gy$}}}%
    \put(0.85622536,0.00444427){\color[rgb]{0,0,0}\makebox(0,0)[lb]{\smash{$y$}}}%
    \put(0.31025434,0.39662145){\color[rgb]{0,0,0}\makebox(0,0)[lb]{\smash{$k$}}}%
    \put(0.07030379,0.2548424){\color[rgb]{0,0,0}\makebox(0,0)[lb]{\smash{$\le \varepsilon$}}}%
    \put(0.18036609,0.09305422){\color[rgb]{0,0,0}\makebox(0,0)[lb]{\smash{$R_0$}}}%
    \put(0.58097396,0.14185046){\color[rgb]{0,0,0}\makebox(0,0)[lb]{\smash{$p$}}}%
    \put(0.55523856,0.44392363){\color[rgb]{0,0,0}\makebox(0,0)[lb]{\smash{$gp$}}}%
    \put(0.38616842,0.41147943){\color[rgb]{0,0,0}\makebox(0,0)[lb]{\smash{$gm$}}}%
    \put(0.36780036,0.3240936){\color[rgb]{0,0,0}\makebox(0,0)[lb]{\smash{$\le 4\delta$}}}%
    \put(0.39096335,0.20708397){\color[rgb]{0,0,0}\makebox(0,0)[lb]{\smash{$\le 4\delta$}}}%
    \put(0.87201095,0.25458065){\color[rgb]{0,0,0}\makebox(0,0)[lb]{\smash{$\le \d(x,y)+\e$}}}%
    \put(0.37812275,0.12281068){\color[rgb]{0,0,0}\makebox(0,0)[lb]{\smash{$m$}}}%
    \put(0.33570931,0.27633517){\color[rgb]{0,0,0}\makebox(0,0)[lb]{\smash{$n$}}}%
  \end{picture}%
\endgroup

%% file: fig1.pdf_tex

\begingroup
  \makeatletter
  \providecommand\color[2][]{%
    \errmessage{(Inkscape) Color is used for the text in Inkscape, but the package 'color.sty' is not loaded}
    \renewcommand\color[2][]{}%
  }
  \providecommand\transparent[1]{%
    \errmessage{(Inkscape) Transparency is used (non-zero) for the text in Inkscape, but the package 'transparent.sty' is not loaded}
    \renewcommand\transparent[1]{}%
  }
  \providecommand\rotatebox[2]{#2}
  \ifx\svgwidth\undefined
    \setlength{\unitlength}{401.94364014pt}
  \else
    \setlength{\unitlength}{\svgwidth}
  \fi
  \global\let\svgwidth\undefined
  \makeatother
  \begin{picture}(1,0.29933655)%
    \put(0,0){\includegraphics[width=\unitlength]{fig1.pdf}}%
    \put(0.06946207,0.1306944){\color[rgb]{0,0,0}\makebox(0,0)[lb]{\smash{$p_1$}}}%
    \put(0.20667397,0.19558863){\color[rgb]{0,0,0}\makebox(0,0)[lb]{\smash{$a_1$}}}%
    \put(0.30421606,0.21655065){\color[rgb]{0,0,0}\makebox(0,0)[lb]{\smash{$p_2$}}}%
    \put(0.40776955,0.20251981){\color[rgb]{0,0,0}\makebox(0,0)[lb]{\smash{$a_2$}}}%
    \put(0.57003444,0.17964036){\color[rgb]{0,0,0}\makebox(0,0)[lb]{\smash{.     .     .}}}%
    \put(0.77025199,0.14077381){\color[rgb]{0,0,0}\makebox(0,0)[lb]{\smash{$a_n$}}}%
    \put(0.86077222,0.17653106){\color[rgb]{0,0,0}\makebox(0,0)[lb]{\smash{$p_{n+1}$}}}%
    \put(0.96444231,0.17381038){\color[rgb]{0,0,0}\makebox(0,0)[lb]{\smash{$g$}}}%
    \put(-0.00100294,0.07081398){\color[rgb]{0,0,0}\makebox(0,0)[lb]{\smash{$f$}}}%
    \put(0.19973865,0.06808799){\color[rgb]{0,0,0}\makebox(0,0)[lb]{\smash{$\;\;\;C_1$}}}%
    \put(0.39222977,0.06749201){\color[rgb]{0,0,0}\makebox(0,0)[lb]{\smash{$\;\;\;C_2 $}}}%
    \put(0.7563537,0.06451229){\color[rgb]{0,0,0}\makebox(0,0)[lb]{\smash{$\;\;C_n$}}}%
  \end{picture}%
\endgroup

%% file: fig2.pdf_tex

\begingroup
  \makeatletter
  \providecommand\color[2][]{%
    \errmessage{(Inkscape) Color is used for the text in Inkscape, but the package 'color.sty' is not loaded}
    \renewcommand\color[2][]{}%
  }
  \providecommand\transparent[1]{%
    \errmessage{(Inkscape) Transparency is used (non-zero) for the text in Inkscape, but the package 'transparent.sty' is not loaded}
    \renewcommand\transparent[1]{}%
  }
  \providecommand\rotatebox[2]{#2}
  \ifx\svgwidth\undefined
    \setlength{\unitlength}{348.46491089pt}
  \else
    \setlength{\unitlength}{\svgwidth}
  \fi
  \global\let\svgwidth\undefined
  \makeatother
  \begin{picture}(1,0.45618518)%
    \put(0,0){\includegraphics[width=\unitlength]{fig2.pdf}}%
    \put(0.59637634,0.22241589){\color[rgb]{0,0,0}\makebox(0,0)[lb]{\smash{$e$}}}%
    \put(0.3839208,0.26736299){\color[rgb]{0,0,0}\makebox(0,0)[lb]{\smash{$r_1$}}}%
    \put(0.67419204,0.40131039){\color[rgb]{0,0,0}\makebox(0,0)[lb]{\smash{$r_2$}}}%
    \put(0.58248919,0.34585724){\color[rgb]{0,0,0}\makebox(0,0)[lb]{\smash{$b$}}}%
    \put(0.57014501,0.14796533){\color[rgb]{0,0,0}\makebox(0,0)[lb]{\smash{$a_2$}}}%
    \put(0.42162979,0.10128907){\color[rgb]{0,0,0}\makebox(0,0)[lb]{\smash{$p_2$}}}%
    \put(0.28237246,0.14372204){\color[rgb]{0,0,0}\makebox(0,0)[lb]{\smash{$a_1$}}}%
    \put(0.12845653,0.08547319){\color[rgb]{0,0,0}\makebox(0,0)[lb]{\smash{$p_1$}}}%
    \put(0.75260684,0.09704578){\color[rgb]{0,0,0}\makebox(0,0)[lb]{\smash{$p_3$}}}%
    \put(0.27118559,0.07428632){\color[rgb]{0,0,0}\makebox(0,0)[lb]{\smash{$\;\;C_j$}}}%
    \put(0.55818666,0.07428625){\color[rgb]{0,0,0}\makebox(0,0)[lb]{\smash{$\;\;C_i$}}}%
    \put(0.95898529,0.10244635){\color[rgb]{0,0,0}\makebox(0,0)[lb]{\smash{$g$}}}%
    \put(0.7788381,0.43689521){\color[rgb]{0,0,0}\makebox(0,0)[lb]{\smash{$h$}}}%
    \put(0.85135986,0.25790528){\color[rgb]{0,0,0}\makebox(0,0)[lb]{\smash{$q$}}}%
    \put(-0.00115686,0.10321783){\color[rgb]{0,0,0}\makebox(0,0)[lb]{\smash{$f$}}}%
  \end{picture}%
\endgroup

%% file: fig07.pdf_tex

\begingroup
  \makeatletter
  \providecommand\color[2][]{%
    \errmessage{(Inkscape) Color is used for the text in Inkscape, but the package 'color.sty' is not loaded}
    \renewcommand\color[2][]{}%
  }
  \providecommand\transparent[1]{%
    \errmessage{(Inkscape) Transparency is used (non-zero) for the text in Inkscape, but the package 'transparent.sty' is not loaded}
    \renewcommand\transparent[1]{}%
  }
  \providecommand\rotatebox[2]{#2}
  \ifx\svgwidth\undefined
    \setlength{\unitlength}{291.03907471pt}
  \else
    \setlength{\unitlength}{\svgwidth}
  \fi
  \global\let\svgwidth\undefined
  \makeatother
  \begin{picture}(1,0.39778227)%
    \put(0,0){\includegraphics[width=\unitlength]{fig07.pdf}}%
    \put(0.04587093,0.06871644){\color[rgb]{0,0,0}\makebox(0,0)[lb]{\smash{$1$}}}%
    \put(-0.00092342,0.37799788){\color[rgb]{0,0,0}\makebox(0,0)[lb]{\smash{$f$}}}%
    \put(0.41072064,0.02484674){\color[rgb]{0,0,0}\makebox(0,0)[lb]{\smash{$u$}}}%
    \put(0.41072064,0.36922394){\color[rgb]{0,0,0}\makebox(0,0)[lb]{\smash{$v$}}}%
    \put(0.93203903,0.07675921){\color[rgb]{0,0,0}\makebox(0,0)[lb]{\smash{$g$}}}%
    \put(0.87647069,0.37580439){\color[rgb]{0,0,0}\makebox(0,0)[lb]{\smash{$fg$}}}%
    \put(0.65639111,0.03215836){\color[rgb]{0,0,0}\makebox(0,0)[lb]{\smash{$p$}}}%
    \put(0.6059409,0.38238485){\color[rgb]{0,0,0}\makebox(0,0)[lb]{\smash{$q$}}}%
    \put(0.09997693,0.20690602){\color[rgb]{0,0,0}\makebox(0,0)[lb]{\smash{$\le \e$}}}%
    \put(0.22719906,0.00144955){\color[rgb]{0,0,0}\makebox(0,0)[lb]{\smash{$3\e+1$}}}%
    \put(0.45459034,0.20690602){\color[rgb]{0,0,0}\makebox(0,0)[lb]{\smash{$\le 2\e$}}}%
    \put(0.93496363,0.22006691){\color[rgb]{0,0,0}\makebox(0,0)[lb]{\smash{$\le \e$}}}%
    \put(-0.00019227,0.20471253){\color[rgb]{0,0,0}\makebox(0,0)[lb]{\smash{$s_1$}}}%
    \put(0.8406438,0.21348645){\color[rgb]{0,0,0}\makebox(0,0)[lb]{\smash{$s_2$}}}%
  \end{picture}%
\endgroup

%% file: fig06.pdf_tex

\begingroup
  \makeatletter
  \providecommand\color[2][]{%
    \errmessage{(Inkscape) Color is used for the text in Inkscape, but the package 'color.sty' is not loaded}
    \renewcommand\color[2][]{}%
  }
  \providecommand\transparent[1]{%
    \errmessage{(Inkscape) Transparency is used (non-zero) for the text in Inkscape, but the package 'transparent.sty' is not loaded}
    \renewcommand\transparent[1]{}%
  }
  \providecommand\rotatebox[2]{#2}
  \ifx\svgwidth\undefined
    \setlength{\unitlength}{387.23739014pt}
  \else
    \setlength{\unitlength}{\svgwidth}
  \fi
  \global\let\svgwidth\undefined
  \makeatother
  \begin{picture}(1,0.44736442)%
    \put(0,0){\includegraphics[width=\unitlength]{fig06.pdf}}%
    \put(0.20942932,0.23005387){\color[rgb]{0,0,0}\makebox(0,0)[lb]{\smash{$C_i$}}}%
    \put(0.45847367,0.34498121){\color[rgb]{0,0,0}\makebox(0,0)[lb]{\smash{$B$}}}%
    \put(0.71097995,0.22714697){\color[rgb]{0,0,0}\makebox(0,0)[lb]{\smash{$C_{i+1}$}}}%
    \put(0.33898915,0.23080514){\color[rgb]{0,0,0}\makebox(0,0)[lb]{\smash{$p_1$}}}%
    \put(0.46572808,0.25658404){\color[rgb]{0,0,0}\makebox(0,0)[lb]{\smash{$d$}}}%
    \put(0.59044705,0.23095298){\color[rgb]{0,0,0}\makebox(0,0)[lb]{\smash{$p_2$}}}%
    \put(0.20420681,0.12731984){\color[rgb]{0,0,0}\makebox(0,0)[lb]{\smash{$f_1$}}}%
    \put(0.73002047,0.12731984){\color[rgb]{0,0,0}\makebox(0,0)[lb]{\smash{$f_2$}}}%
    \put(0.21436812,0.04728589){\color[rgb]{0,0,0}\makebox(0,0)[lb]{\smash{$e_1$}}}%
    \put(0.10018007,0.05976281){\color[rgb]{0,0,0}\makebox(0,0)[lb]{\smash{$q_1$}}}%
    \put(0.45757469,0.02927882){\color[rgb]{0,0,0}\makebox(0,0)[lb]{\smash{$q_2$}}}%
    \put(0.73180659,0.04609115){\color[rgb]{0,0,0}\makebox(0,0)[lb]{\smash{$e_2$}}}%
    \put(0.85547867,0.06095754){\color[rgb]{0,0,0}\makebox(0,0)[lb]{\smash{$q_3$}}}%
    \put(-0.00069402,0.10631988){\color[rgb]{0,0,0}\makebox(0,0)[lb]{\smash{$f$}}}%
    \put(0.97539462,0.10258782){\color[rgb]{0,0,0}\makebox(0,0)[lb]{\smash{$g$}}}%
  \end{picture}%
\endgroup

%% file: fig04.pdf_tex
\begingroup%
  \makeatletter%
  \providecommand\color[2][]{%
    \errmessage{(Inkscape) Color is used for the text in Inkscape, but the package 'color.sty' is not loaded}%
    \renewcommand\color[2][]{}%
  }%
  \providecommand\transparent[1]{%
    \errmessage{(Inkscape) Transparency is used (non-zero) for the text in Inkscape, but the package 'transparent.sty' is not loaded}%
    \renewcommand\transparent[1]{}%
  }%
  \providecommand\rotatebox[2]{#2}%
  \ifx\svgwidth\undefined%
    \setlength{\unitlength}{372.44132385bp}%
    \ifx\svgscale\undefined%
      \relax%
    \else%
      \setlength{\unitlength}{\unitlength * \real{\svgscale}}%
    \fi%
  \else%
    \setlength{\unitlength}{\svgwidth}%
  \fi%
  \global\let\svgwidth\undefined%
  \global\let\svgscale\undefined%
  \makeatother%
  \begin{picture}(1,0.50474304)%
    \put(0,0){\includegraphics[width=\unitlength]{fig04.pdf}}%
    \put(-0.00076564,0.17295946){\color[rgb]{0,0,0}\makebox(0,0)[lb]{\smash{$1$}}}%
    \put(0.00805121,0.48582008){\color[rgb]{0,0,0}\makebox(0,0)[lb]{\smash{$f$}}}%
    \put(0.05256417,0.32683444){\color[rgb]{0,0,0}\makebox(0,0)[lb]{\smash{$s$}}}%
    \put(0.11699478,0.13908737){\color[rgb]{0,0,0}\makebox(0,0)[lb]{\smash{$p_1$}}}%
    \put(0.12533202,0.48795139){\color[rgb]{0,0,0}\makebox(0,0)[lb]{\smash{$q_1$}}}%
    \put(0.58290816,0.40681004){\color[rgb]{0,0,0}\makebox(0,0)[lb]{\smash{$b$}}}%
    \put(0.78151825,0.48460445){\color[rgb]{0,0,0}\makebox(0,0)[lb]{\smash{$q_2$}}}%
    \put(0.9591903,0.48460445){\color[rgb]{0,0,0}\makebox(0,0)[lb]{\smash{$fg$}}}%
    \put(0.95723864,0.1874601){\color[rgb]{0,0,0}\makebox(0,0)[lb]{\smash{$g$}}}%
    \put(0.81190646,0.14218615){\color[rgb]{0,0,0}\makebox(0,0)[lb]{\smash{$p_2$}}}%
    \put(0.25573581,0.12406455){\color[rgb]{0,0,0}\makebox(0,0)[lb]{\smash{$a$}}}%
    \put(0.5594454,0.16062439){\color[rgb]{0,0,0}\makebox(0,0)[lb]{\smash{$f^{-1}b$}}}%
    \put(0.43242272,0.15487212){\color[rgb]{0,0,0}\makebox(0,0)[lb]{\smash{$f^{-1}v$}}}%
    \put(0.20486328,0.17038283){\color[rgb]{0,0,0}\makebox(0,0)[lb]{\smash{$u$}}}%
    \put(0.42759538,0.27012442){\color[rgb]{0,0,0}\makebox(0,0)[lb]{\smash{$e$}}}%
    \put(0.53727493,0.37036225){\color[rgb]{0,0,0}\makebox(0,0)[lb]{\smash{$C$}}}%
    \put(0.73007355,0.30919343){\color[rgb]{0,0,0}\makebox(0,0)[lb]{\smash{$f^{-1}C$}}}%
    \put(0.51588484,0.45131615){\color[rgb]{0,0,0}\makebox(0,0)[lb]{\smash{$v$}}}%
  \end{picture}%
\endgroup%

%% file: fig05.pdf_tex
\begingroup%
  \makeatletter%
  \providecommand\color[2][]{%
    \errmessage{(Inkscape) Color is used for the text in Inkscape, but the package 'color.sty' is not loaded}%
    \renewcommand\color[2][]{}%
  }%
  \providecommand\transparent[1]{%
    \errmessage{(Inkscape) Transparency is used (non-zero) for the text in Inkscape, but the package 'transparent.sty' is not loaded}%
    \renewcommand\transparent[1]{}%
  }%
  \providecommand\rotatebox[2]{#2}%
  \ifx\svgwidth\undefined%
    \setlength{\unitlength}{388.71424283bp}%
    \ifx\svgscale\undefined%
      \relax%
    \else%
      \setlength{\unitlength}{\unitlength * \real{\svgscale}}%
    \fi%
  \else%
    \setlength{\unitlength}{\svgwidth}%
  \fi%
  \global\let\svgwidth\undefined%
  \global\let\svgscale\undefined%
  \makeatother%
  \begin{picture}(1,0.3803987)%
    \put(0,0){\includegraphics[width=\unitlength]{fig05.pdf}}%
    \put(0.0431339,0.05398163){\color[rgb]{0,0,0}\makebox(0,0)[lb]{\smash{$1$}}}%
    \put(0.04957761,0.36226792){\color[rgb]{0,0,0}\makebox(0,0)[lb]{\smash{$f$}}}%
    \put(0.0279951,0.11692512){\color[rgb]{0,0,0}\makebox(0,0)[lb]{\smash{$s_1$}}}%
    \put(0.15396042,0.03005064){\color[rgb]{0,0,0}\makebox(0,0)[lb]{\smash{$p_1$}}}%
    \put(0.16194864,0.36431001){\color[rgb]{0,0,0}\makebox(0,0)[lb]{\smash{$q_1$}}}%
    \put(0.60036906,0.28656551){\color[rgb]{0,0,0}\makebox(0,0)[lb]{\smash{$b$}}}%
    \put(0.79066465,0.36110318){\color[rgb]{0,0,0}\makebox(0,0)[lb]{\smash{$q_2$}}}%
    \put(0.96089873,0.36110318){\color[rgb]{0,0,0}\makebox(0,0)[lb]{\smash{$fg$}}}%
    \put(0.95902878,0.07639832){\color[rgb]{0,0,0}\makebox(0,0)[lb]{\smash{$g$}}}%
    \put(0.8197807,0.03301969){\color[rgb]{0,0,0}\makebox(0,0)[lb]{\smash{$p_2$}}}%
    \put(0.28689328,0.01565673){\color[rgb]{0,0,0}\makebox(0,0)[lb]{\smash{$a$}}}%
    \put(0.23815045,0.06003597){\color[rgb]{0,0,0}\makebox(0,0)[lb]{\smash{$u$}}}%
    \put(0.49510987,0.13237447){\color[rgb]{0,0,0}\makebox(0,0)[lb]{\smash{$e$}}}%
    \put(0.53615157,0.32920845){\color[rgb]{0,0,0}\makebox(0,0)[lb]{\smash{$v$}}}%
    \put(0.0368076,0.20134264){\color[rgb]{0,0,0}\makebox(0,0)[lb]{\smash{$d$}}}%
    \put(0.03480355,0.30154478){\color[rgb]{0,0,0}\makebox(0,0)[lb]{\smash{$s_2$}}}%
    \put(0.39252523,0.14322539){\color[rgb]{0,0,0}\makebox(0,0)[lb]{\smash{$C$}}}%
    \put(0.26326445,0.13821526){\color[rgb]{0,0,0}\makebox(0,0)[lb]{\smash{$t_1$}}}%
    \put(0.34142213,0.21236487){\color[rgb]{0,0,0}\makebox(0,0)[lb]{\smash{$t_2$}}}%
  \end{picture}%
\endgroup%

%% file: fig08.pdf_tex
\begingroup%
  \makeatletter%
  \providecommand\color[2][]{%
    \errmessage{(Inkscape) Color is used for the text in Inkscape, but the package 'color.sty' is not loaded}%
    \renewcommand\color[2][]{}%
  }%
  \providecommand\transparent[1]{%
    \errmessage{(Inkscape) Transparency is used (non-zero) for the text in Inkscape, but the package 'transparent.sty' is not loaded}%
    \renewcommand\transparent[1]{}%
  }%
  \providecommand\rotatebox[2]{#2}%
  \ifx\svgwidth\undefined%
    \setlength{\unitlength}{270.52345792bp}%
    \ifx\svgscale\undefined%
      \relax%
    \else%
      \setlength{\unitlength}{\unitlength * \real{\svgscale}}%
    \fi%
  \else%
    \setlength{\unitlength}{\svgwidth}%
  \fi%
  \global\let\svgwidth\undefined%
  \global\let\svgscale\undefined%
  \makeatother%
  \begin{picture}(1,0.40073319)%
    \put(0,0){\includegraphics[width=\unitlength]{fig08.pdf}}%
    \put(0.00935728,0.06159258){\color[rgb]{0,0,0}\makebox(0,0)[lb]{\smash{$1$}}}%
    \put(0.01490945,0.37464758){\color[rgb]{0,0,0}\makebox(0,0)[lb]{\smash{$f$}}}%
    \put(0.12445481,0.0264404){\color[rgb]{0,0,0}\makebox(0,0)[lb]{\smash{$p_1$}}}%
    \put(0.13027161,0.38455057){\color[rgb]{0,0,0}\makebox(0,0)[lb]{\smash{$q_1$}}}%
    \put(0.57986371,0.28575097){\color[rgb]{0,0,0}\makebox(0,0)[lb]{\smash{$e$}}}%
    \put(0.76104029,0.3786336){\color[rgb]{0,0,0}\makebox(0,0)[lb]{\smash{$q_2$}}}%
    \put(0.94996077,0.37641733){\color[rgb]{0,0,0}\makebox(0,0)[lb]{\smash{$fg$}}}%
    \put(0.94427725,0.06908214){\color[rgb]{0,0,0}\makebox(0,0)[lb]{\smash{$g$}}}%
    \put(0.79162402,0.02799436){\color[rgb]{0,0,0}\makebox(0,0)[lb]{\smash{$p_2$}}}%
    \put(0.2865217,0.00213275){\color[rgb]{0,0,0}\makebox(0,0)[lb]{\smash{$d$}}}%
    \put(0.21872156,0.07359258){\color[rgb]{0,0,0}\makebox(0,0)[lb]{\smash{$u$}}}%
    \put(0.41659963,0.15336256){\color[rgb]{0,0,0}\makebox(0,0)[lb]{\smash{$c$}}}%
    \put(0.52094011,0.33888226){\color[rgb]{0,0,0}\makebox(0,0)[lb]{\smash{$v$}}}%
    \put(-0.00073786,0.21394014){\color[rgb]{0,0,0}\makebox(0,0)[lb]{\smash{$s_1$}}}%
    \put(0.94455067,0.21212379){\color[rgb]{0,0,0}\makebox(0,0)[lb]{\smash{$s_2$}}}%
  \end{picture}%
\endgroup%